\documentclass[12 pt]{amsart}
\usepackage{amssymb,amsfonts,amsthm}
\usepackage{graphicx}
\usepackage{bm}
\usepackage{float, placeins}
\restylefloat{table}
\newtheorem{thm}{Theorem}[section]

\newtheorem{pro}[thm]{Proposition}

\theoremstyle{definition}

\theoremstyle{remark}

\numberwithin{equation}{section}

\usepackage[normalem]{ulem}
\usepackage{hyperref}

\newcommand{\com}[1]{ { } }

\begin{document}
\title[Data-driven Fokker-Planck solver]{An efficient data-driven solver for
Fokker-Planck equations: algorithm and analysis.}
\author{Matthew Dobson}
\address{Matthew Dobson: Department of Mathematics and Statistics, University
  of Massachusetts Amherst, Amherst, MA, 01002, USA}
\email{dobson@math.umass.edu}
\author{Yao Li}
\address{Yao Li: Department of Mathematics and Statistics, University
  of Massachusetts Amherst, Amherst, MA, 01002, USA}
\email{yaoli@math.umass.edu}
\author{Jiayu Zhai}
\address{Jiayu Zhai: Department of Mathematics and Statistics, University
  of Massachusetts Amherst, Amherst, MA, 01002, USA}
\email{zhai@math.umass.edu}

\thanks{Yao Li is partially supported by NSF DMS-1813246.}

\keywords{Fokker-Planck equation, Monte Carlo simulation, data-driven
  method}

\begin{abstract}
  Computing the invariant probability measure of a randomly perturbed
  dynamical system usually means solving the stationary Fokker-Planck equation. This paper studies several key properties of a novel
  data-driven solver for low-dimensional Fokker-Planck equations proposed in
  \cite{li2018data}. Based on these results, we propose a new ``block
  solver'' for the stationary Fokker-Planck equation, which
  significantly improves the performance of the original
  algorithm. Some possible ways of reducing numerical artifacts caused
  by the block solver are discussed and tested with examples. 
\end{abstract}
\maketitle
\section{Introduction}
Random perturbations to deterministic dynamical systems are ubiquitous in models used in physics, biology and
engineering. The steady state of a randomly perturbed dynamical system is of critical interest in the study
of these physical, biological or chemical systems and their
applications. From a dynamical systems point of view, the
  interplay of dynamics and noise is both interesting and
  challenging, especially if the underlying dynamics is
  chaotic. Characteristics of the steady state distribution also help
  us to understand asymptotic effects of random perturbations to
  deterministic dynamics.

 The evolution of the probability density function of a randomly perturbed
system is described by the Fokker-Planck equation \cite{risken1996fokker}. Consider a stochastic dynamical system 
\begin{equation}
\label{SDE}
  \mathrm{d}X_{t} = f(X_{t}) \mathrm{d}t + \sigma(X_{t})
  \mathrm{d}W_{t} \,,
\end{equation}
where $f$ is a vector field in $\mathbb{R}^{n}$, $\sigma$ is a coefficient
matrix, and $\mathrm{d}W_{t}$ is an $n$-dimensional white noise. The
corresponding Fokker-Planck equation, which is also known as the
Kolmogorov forward equation, is
\begin{equation}
\label{FPE}
u_t = \mathcal{L}u = -\sum_{i = 1}^{n} (f_{i}u)_{x_{i}} +
  \frac{1}{2}\sum_{i,j = 1}^{n}(D_{i,j}u)_{x_{i} x_{j}} \,,
\end{equation}
where $D = \sigma^{T}\sigma$, $u(x,t)$ denotes the probability density at time $t$, and subscripts 
$t$ and $x_i$ denote partial derivatives. In this paper, we focus on
the invariant probability measure of \eqref{SDE}, whose density
function satisfies the stationary Fokker-Planck equation
$$\mathcal{L}u = 0 \qquad \int_\Omega u \, dx = 1.$$
Detailed assumptions about equation \eqref{SDE} and \eqref{FPE} will
be given in Section \ref{Algorithm description}.

For Langevin dynamics, the invariant probability measure is given by the Gibbs distribution
which can be computed up to the unknown normalizing constant; however, in general,
the Fokker-Planck equation can not be solved
analytically. Rigorous estimations of the invariant probability density function are
challenging as well. Most known results are proved by large deviations
techniques \cite{freidlin1998random}, which unfortunately only shows
tail properties when the noise is asymptotically small. Some concentration properties
of the invariant probability measure can be proved by assuming some
dissipative conditions. For example, it was shown in
\cite{li2016systematic, day1985some} that such
concentration in the vicinity of a strong attractor is ``Gaussian-like''. However, these theoretical results can rarely
give a satisfactory quantitative description of the invariant
probability measure. Therefore, numerical solution techniques are
necessary to further study these randomly perturbed dynamical
systems. Numerically solving a steady state Fokker-Planck equation in
an unbounded domain is nontrivial. And additional challenges are presented in systems with high 
dimensionality, chaotic underlying dynamics, and multiscale
coefficient terms.


One difficulty of solving the Fokker-Planck equation numerically is
the conflict between the need for high-resolution local solutions and the
necessity to handle large spatial domains. On
one hand, in many applications, what we need is a high-resolution
local numerical solution. It is known that the invariant probability
measure tends to concentrate at the vicinity of the global
attractor, and for such systems we are interested in the
distribution in a local region of the phase space. In addition, if the strength of noise $\sigma$ is
small, it is proved that the probability density function is
concentrated in an $O(\sigma)$ neighborhoods of the global attractor\cite{li2016systematic}. So in
order to obtain a meaningful solution and to avoid numerical
artifacts, the grid size needs to be small enough. On the other hand, the Fokker-Planck equation in $\mathbb{R}^{d}$ is defined
on an unbounded domain with zero value at infinity. The lack
of a local boundary condition makes the problem computationally
challenging. The existing methods usually solve the Fokker-Planck equation in a region that is large enough to cover all attractors. 


In \cite{li2018data}, a hybrid method is proposed to partially resolve the
difficulties. The method deals with the local Fokker-Planck equation
and completely removes the unknown boundary condition, which makes the
resultant linear system undetermined. To solve this
underdetermined problem, Monte-Carlo
simulation is used to provide a reference solution for the numerical
solver (finite difference or finite element). The reference solution
itself has low accuracy and lots of undesired fluctuations. The algorithm then projects the
``noisy'' reference solution onto the kernel of the discretized numerical solver. This
minimizes the distance between the collection of possible numerical
solutions (without knowing the boundary condition) and the
reference solution from the Monte Carlo
simulation. This method can solve the problem in any local area even
if it doesn't cover any attractor. It also smooths the oscillation
caused by the Monte Carlo sampling. By reducing the computational cost
from non-locality, it can provide a high resolution solution in a
local area.

Paper \cite{li2018data} only introduced the algorithm without
proof. Analysis of this algorithm is carried out in this paper. We proved that the
hybrid method introduced in \cite{li2018data} can significantly reduce the
error of the reference solution produced by the Monte Carlo
simulation. The heuristic reason is that the error term of this random reference
solution is very close to an i.i.d. random vector. The expected
norm of this random vector is dramatically reduced when projecting it
to a lower dimensional subspace. In addition, we use a combination of rigorous analysis
and numerical computations to show that the error term of this projection
concentrates on the boundary of the numerical domain. In other words,
the empirical performance of the hybrid algorithm is actually much
better than what can be rigorously proved.

The other goal of this paper is to improve the performance of this hybrid method by
introducing block solvers. This improvement is motivated by the locality of
the hybrid method. Since the hybrid method does not rely on local boundary
conditions, we can divide the numerical domain into a large number of small
blocks and apply the hybrid method to each block. The global solution is
a collage of local solutions on these blocks. This divide-and-conquer
strategy is very efficient. Consider a $d$-dimensional problems with
$N$ grid points in each dimension. The classical numerical PDE solver
needs to solve a large linear system with $N^{d}$ variables. Assume
the cost of solving a linear system with $n$ variables is
$O(n^{p})$. Then the total cost is $O(N^{dp})$, which is considerably
large if for instance $d = 3$ and $N =  1000$. However, if we divide the grid into many blocks with only $m$
grid points in each dimension. The total cost of solving the
Fokker-Planck equation on $(N/m)^{d}$ blocks becomes
$m^{pd}\times (N/m)^{d} = m^{(p-1)d}N^{d}$. Empirically $m$ can be as small
as $20-30$. This dramatically reduces the total computational cost,
unless the linear solver can achieve a linear complexity (which usually does not happen). In
addition, parallelizing these block solvers is much easier than computing
a large linear system in parallel. Instead of a local solution in a
small subset of the phase space demonstrated in \cite{li2018data}, the
block solver now allows us to compute the full invariant probability
density function of 3D or 4D systems, as demonstrated later in this paper. 
 
The idea of using local blocks is supported by our analytical results in
  the first half of this paper. Theoretically, using larger blocks gives
  better reduction of error terms from Monte Carlo simulations, as
  proved in Theorem \ref{proj}. But
  the analysis in this paper shows that the error tends to concentrate at the
  boundary of blocks. Hence the size of blocks needs not to be very
  large to make the accuracy of solutions in the interior of blocks
  acceptable. And the error on the boundary can be repaired by
  algorithms. Since the error of numerical solution mainly
  concentrates on the boundary, a naive block solver has visible
  interface errors between blocks. We then develop methods to reduce this interface
  error. Two different approaches, namely the overlapping
blocks method and the shifting blocks method, are introduced and
tested with several examples.

In this paper, we mainly consider low-dimensional systems up to
dimension 3 or 4, where traditional grid-based numerical methods still
work. For systems in much higher dimensions, all traditional grid-based methods of
solving the Fokker-Planck equation, such as finite difference method or finite elements method, are not feasible
any more. Direct Monte Carlo simulation also greatly suffers from the
curse-of-dimensionality.  There are several techniques
introduced to deal with certain multi-dimensional Fokker-Planck equations, such as
the truncated asymptotic expansion, splitting method, orthogonal
functions, and tensor decompositions \cite{er2011methodology,
  er2012state, majda2001mathematical, von2000calculation, sun2015numerical}. In particular, \cite{chen2017beating,
  chen2018efficient}
introduced an efficient technique for a class of high-dimensional
dynamical systems. In the future, we will incorporate these high-dimensional
sampling techniques to the mesh-free version of this hybrid
algorithm.

The rest of this paper is organized as follows. In Section
\ref{Analysis of data-driven Fokker-Planck solver}, we review the
hybrid method in \cite{li2018data} and rigorously analyze the convergence of
the method. We also show that the error will concentrate on the
boundary of the domain. A directed block solver in
proposed in Section \ref{Block Fokker-Planck Solver}. Two possible methods to
repair interface error between blocks are studied in Section \ref{Reducing
  Interface Error}. In Section \ref{Numerical Examples}, we use three
example systems to test our algorithms and error reduction
methods. Section \ref{conclusion} is the conclusion.

\section{Analysis of data-driven Fokker-Planck solver}\label{Analysis of data-driven Fokker-Planck solver}
\subsection{Algorithm description}\label{Algorithm description}
Consider a stochastic differential equation
\begin{equation}
\label{SDE1}
  \mathrm{d}X_{t} = f(X_{t}) \mathrm{d}t + \sigma(X_{t})
  \mathrm{d}W_{t} \,,
\end{equation}
where $X_{t} \in \mathbb{R}^{n}$, $f : \mathbb{R}^{n} \rightarrow
\mathbb{R}^{n}$ is a continuous vector field, $\sigma$ is a $n\times
n$ matrix-valued function, and $\mathrm{d}W_{t}$ is the white noise in
$\mathbb{R}^{n}$. We assume that $f$ and $\sigma$ has enough
regularity such that equation \eqref{SDE1} admits a unique solution
$X_{t}$ that is a Markov process with a transition kernel
$P^{t}(x, \cdot)$. Similar as in \cite{li2018data}, we further assume that $X_{t}$
admits a unique invariant probability measure $\pi$ such that 
$$
  \pi P^{t}(A) = \int_{\mathbb{R}^{n} } P^{t}(x, A) \pi( \mathrm{d}x)
$$
for any measurable set $A$. In addition, we assume $\pi$ is absolutely
continuous with respect to the Lebesgue measure, and $P^{t}(x, \cdot)$
converges to $\pi$ for any $x \in \mathbb{R}^{n}$. We refer
\cite{oksendal2003stochastic, karatzas2012brownian, zeeman1988stability, bogachev2001generalization,
  huang2015steady, khasminskii2011stochastic, meyn2012markov, hairer2006ergodicity}
for the detailed conditions that lead to the existence of solutions of
\eqref{SDE1}, the existence of an invariant probability measure, and the
convergence to the invariant probability measure.

Let $u$ be the probability density function of $\pi$. It is well known
that $u$ satisfies the stationary Fokker-Planck equation
\begin{equation}\label{F-P}
0 = \mathcal{L}u = -\sum_{i = 1}^{n} (f_{i}u)_{x_{i}} +
  \frac{1}{2}\sum_{i,j = 1}^{n}(D_{i,j}u)_{x_{i} x_{j}} \,,
\end{equation}
where $D = \sigma^{T}\sigma$. In addition, because of the convergence,
we have
$$
  u(x) = \lim_{T \rightarrow \infty}\frac{1}{T}\int_{0}^{T} u(x,t)
  \mathrm{d}t \,,
$$
where $u(x,t)$ is the probability density function of $X_{t}$.

For the sake of simplicity we assume $n = 2$ when introducing the
algorithm. But our algorithm works for any dimension. Now assume
that we would like to solve $u$ numerically on a 2D domain $D = [a_{0},
b_{0}] \times [a_{1}, b_{1}]$. To do this, an $N \times M$ grid is constructed on
$D$ with grid size $h = (b_{0} - a_{0})/N = (b_{1} - a_{1})/M$. Since
$u$ is the density function, we approximate it at the center of each
of the $N \times M$ boxes $\{O_{i,j}\}_{i =
  1, j = 1}^{i = N, j = M}$ with $O_{i,j} = [a_{0} + (i-1)h, a_{0} +
ih] \times [a_{1} + (j-1)h, a_{1} + jh]$. Let
$\mathbf{u} = \{u_{i,j}\}_{i = 1, j = 1}^{i = N, j = M}$ be this
numerical solution on $D$ that we are interested in. $\mathbf{u}$ can
be considered as a vector in $\mathbb{R}^{NM}$. Throughout this paper,
we still denote this
vector by $\mathbf{u}$ when it does not lead to confusion. An entry of
$\mathbf{u}$, denoted by $u_{i,j}$, approximates the probability
density function $u$ at the center of the $(i,j)$-box with coordinate
$(ih + a_{0} - h/2, jh + a_{1} - h/2)$. Now, 
we consider $u$ as the solution to the boundary-free PDE
\eqref{F-P} 
and discretize the operator $\mathcal{L}$ on $D$ with respect to all
$(N-2)(M-2)$ interior boxes. The
discretization of the Fokker-Planck equation with respect to each center point gives a
linear relation among $\{u_{i,j}\}$. This produces a linear constraint
for $\mathbf{u}$, denoted as 
$$
  \mathbf{A} \mathbf{u} = \mathbf{0} \,,
$$
where $\mathbf{A}$ is an $(N-2)(M-2) \times (NM)$ matrix. $\mathbf{A}$
is said to be the discretized Fokker-Planck operator. 

Then we need the Monte Carlo simulation to produce a reference
  solution. Let $\{\mathbf{X}_{n}\}_{n = 1}^{\mathbf{N}}$ be a long
numerical trajectory of the time-$\delta$ sample chain of $X_{t}$,
i.e., $\mathbf{X}_{n} = X_{n \delta}$, where $\delta > 0$ is the time step
size of the Monte Carlo simulation. Let $\mathbf{v}
= \{v_{i,j}\}_{i = 1, j = 1}^{i = N, j = M}$ such that
$$
  v_{i,j} = \frac{1}{\mathbf{N} h^{2}} \sum_{n = 1}^{\mathbf{N}}
  \mathbf{1}_{O_{i,j}}(X_{n}) \,.
$$ 
It follows from the ergodicity of \eqref{SDE1} that $\mathbf{v}$ is an
approximate solution of \eqref{F-P}. Again, we denote the $N\times M$
vector reshaped from $\mathbf{v}$ by $\mathbf{v}$ as well. 

As introduced in \cite{li2018data}, we look for the solution of the
following optimization problem
\begin{eqnarray}
\label{opt}
 \mbox{min} &  & \| \mathbf{u} - \mathbf{v} \|_{2} \\\nonumber
\mbox{subject to } & & \mathbf{A} \mathbf{u} = \mathbf{0}  \,.
\end{eqnarray}
This is called the least norm problem. Vector
\begin{equation}\label{sol_opt}
  \mathbf{u} = \mathbf{A}^{T}( \mathbf{A}
  \mathbf{A}^{T})^{-1}(-\mathbf{A} \mathbf{v})+\mathbf{v}
\end{equation}
solves the optimization problem \eqref{opt}.

\subsection{Error analysis through projections}
The aim of this section is to show that the solution $\mathbf{u}$ to
the optimization problem \eqref{opt} is a good approximation of the global
analytical solution $u$ on $\mathbb{R}^2$. Let $\mathbf{u}^{\text{ext}} =
\{u^{\text{ext}}_{i,j}\} = u(ih + a_{0} - h/2, jh + a_{1} - h/2)$ be the values
of the exact solution $u$ at the centres of the boxes. We assume that
the Monte Carlo simulation produces an unbiased sample $\mathbf{v}$
that approximates 
$\mathbf{u}^{\text{ext}}$. We note that this assumption is usually not exactly
satisfied because the invariant probability measure of the numerical
scheme that produces $\{\mathbf{X}_{n}\}$ is only an approximation of $\pi$. We refer
\cite{lelievre2016partial} for known results about the difference between the two
invariant measures for Langevin dynamics and
\cite{mattingly2010convergence} for that of generic stochastic
differential equations. However, when $\mathbf{N}$ is large (at least $10^{7} \sim
10^{8}$ in our simulations), $\mathbf{v} - \mathbf{u}^{\text{ext}}$ is usually
``noisy'' enough to be treated as a vector of i.i.d. random
numbers. Improving the quality of sampling is extremely important to
this algorithm. We will address sampling methods in our subsequent work.

In order to make the
rigorous proof, we need the following assumption. 

{\bf (H)} 
\begin{itemize}
  \item[(a)] For $i=1, \dots, N, j=1,\dots,M$, $v_{i,j} - u^{\text{ext}}_{i,j}$ are
i.i.d random variables with expectation $0$ and variance $\zeta^{2}$.
\item[(b)] The finite difference scheme for equation \eqref{F-P} is
  convergent for the boundary value problem on $[a_{0}, b_{0}] \times
  [a_{1}, b_{1}]$ with $L^{\infty}$ error $O(h^{p})$.
\end{itemize}

The performance of the algorithm is measured by $h \mathbb{E}[
\|\mathbf{u} - \mathbf{u}^{\text{ext}}\|_{2}]$, which is the $L^{2}$
numerical integration of the error term. Before solving the
optimization problem \ref{opt}, we have $h \mathbb{E}[
\|\mathbf{v} - \mathbf{u}^{\text{ext}}\|_{2}] = O(\zeta h N) = O(\zeta)$.

\begin{thm}\label{proj}
Assume {\bf (H)} holds. We have the following bound for the $L^{2}$ error
$$
  h \mathbb{E}[\| \mathbf{u} - \mathbf{u}^{{\rm ext}}\|_{2}]\leq O(h^{1/2}\zeta) +
  O(h^{p}) \,.
$$
\end{thm}
\begin{proof}
In order to proceed, we need an auxiliary vector that satisfies the
linear constraint in equation \eqref{opt}. Consider the Fokker-Planck
equation on the extended domain $\tilde{D} = [a_{0}-h, b_{0}+h] \times
[a_{1}-h, b_{1}+h]$ with boundary condition 
\begin{equation}\label{F-P2}
\left\{
\begin{array}{ll}
\mathcal{L}w=0 & (x,y)\in\tilde{D}\\
w(x,y)=u(x,y) & (x,y)\in\partial\tilde{D}
\end{array}
\right.,
\end{equation}
where $h$ is the mesh size. This problem is well-posed and has a
unique solution $u(x,y), (x,y)\in\tilde{D}$. 

Consider the discretization of \eqref{F-P2} by finite difference method. It is of the following form
$$
\left[\begin{matrix}
    \mathbf{A} & \mathbf{0} \\
    \mathbf{B} & \mathbf{C}
  \end{matrix}\right]
  \left[\begin{matrix}
    \mathbf{u}^{\text{lin}}\\
    \mathbf{u_{0}}
  \end{matrix}\right]
  =\left[\begin{matrix}
    \mathbf{0} \\
    \mathbf{0}
  \end{matrix}\right],
$$
where the extended equations
$$
\left[\begin{matrix}
    \mathbf{B} & \mathbf{C}
  \end{matrix}\right]
  \left[\begin{matrix}
    \mathbf{u}^{\text{lin}}\\
    \mathbf{u_{0}}
  \end{matrix}\right]
  =\mathbf{0}
$$
are the equations for the variables on the boundary $\partial D$ of
$D$, and $\mathbf{u_{0}}$ are the values of $u(x,y)$ at grid points on
the boundary $\partial\tilde{D}$ of $\tilde{D}$. This boundary value
problem gives a solution $\mathbf{u}^{\text{lin}}$ that satisfies the linear
constraint. By assumption {\bf
  (H)}, we have
$$
  \| \mathbf{u}^{\text{lin}} - \mathbf{u}^{\text{ext}}\|_{\infty} = O(h^{p}) \,.
$$

By the triangle inequality, it is sufficient to estimate
$$
  \| \mathbf{u} - \mathbf{u}^{\text{lin}} \|_{2} \,.
$$
Let $P$ be the projection matrix to $\mathrm{Ker}(\mathbf{A})$. Then equation
\eqref{opt} implies $\mathbf{u} = P \mathbf{v}$. Since
$\mathbf{u}^{\text{lin}} \in \mathrm{Ker}(A)$, we have
$$
  \mathbf{u} - \mathbf{u}^{\text{lin}} = P \mathbf{v} - \mathbf{u}^{\text{lin}} =
  P( \mathbf{v} - \mathbf{u}^{\text{lin}}) = P( \mathbf{v} -
  \mathbf{u}^{\text{ext}}) + P( \mathbf{u}^{\text{ext}} - \mathbf{u}^{\text{lin}}) \,.
$$
Take the $L^{2}$ norm on both side and apply the triangle inequality, we have
\begin{equation}
\label{2-2-1}
  h\|\mathbf{u} - \mathbf{u}^{\text{lin}}\|_{2} \leq h\|P( \mathbf{v} -
  \mathbf{u}^{\text{ext}})\|_{2} + h\| P( \mathbf{u}^{\text{ext}} - \mathbf{u}^{\text{lin}})
  \|_{2} \,.
\end{equation}
The second term is easy to bound because
\begin{eqnarray}
\label{2-2-2}
  h\| P( \mathbf{u}^{\text{ext}} - \mathbf{u}^{\text{lin}})
  \|_{2}  &\leq& h \| \mathbf{u}^{\text{ext}} - \mathbf{u}^{\text{lin}}\|_{2} \\ \nonumber
&\leq&
  h N \| \mathbf{u}^{\text{ext}}
  - \mathbf{u}^{\text{lin}}\|_{\infty} = O(h^{p}) \,.
\end{eqnarray}

By assumption {\bf (H)}, $\mathbf{w} = \mathbf{v} - \mathbf{u}^{\text{ext}}$
is a random vector with i.i.d. entries. And $P$ projects $\mathbf{w}$
from $\mathbb{R}^{N\times N}$ to $\mathrm{Ker}(\mathbf{A})$. Note that the
dimension of $\mathrm{Ker}(\mathbf{A})$ is $4N - 4$. Let $S \in SO(N^{2})$ be an
orthogonal matrix such that the first $4N-4$ columns of $S^{T}$ form an
orthonormal basis of $\mathrm{Ker}(\mathbf{A})$. Let $\mathbf{s}_{1}, \cdots,
\mathbf{s}_{N^{2}}$ be column vectors of $S^{T}$. Then $S$ is a
change-of-coordinate matrix such that $\mathrm{Ker}(\mathbf{A})$ is spanned by
$\mathbf{s}_{1}, \cdots, \mathbf{s}_{4N-4}$.

Let
$$
  S \mathbf{w} = [\hat{w}_{1}, \cdots, \hat{w}_{N^{2}}]^{T} \,.
$$
We have
$$
  P \mathbf{w} = \sum_{i = 1}^{4N-4} \hat{w}_{i} \mathbf{s}_{i} \,.
$$
This implies 
$$
  \mathbb{E}[ \|P \mathbf{w}\|_{2}] = \mathbb{E}[ \left (\sum_{i =
      1}^{4N-4}\hat{w}_{i}^{2} \right )^{1/2}] \leq \left (\sum_{i = 1}^{4N-4}
  \mathbb{E}[\hat{w}_{i}^{2}] \right )^{1/2}\,,
$$
because $\{\mathbf{s}_{i = 1}^{N^{2}}\}$ are orthonormal
vectors. 

We have
$$
  \hat{w}_{i} = \sum_{j = 1}^{N^{2}}S_{ji} w_{i} \,,
$$
where $w_{i}$ is the $i$-th entry of $\mathbf{w}$. $S$ is orthogonal
hence
$$
  \sum_{j = 1}^{N^{2}} S_{ji}^{2} = 1 \,. 
$$
Recall that entries of $\mathbf{w}$ are i.i.d. random
variables with expectation zero and variance $\zeta^{2}$. This implies
$$
  \mathbb{E}[ \hat{w}_{i}^{2}] = \zeta^{2}\sum_{j = 1}^{N^{2}}
  S_{ji}^{2} = \zeta^{2} \,.
$$
Hence
\begin{equation}
\label{2-2-3}
  \mathbb{E}[ \| \mathbf{v} - \mathbf{u}^{\text{ext}} \|_{2}] =
  \mathbb{E}[\| P \mathbf{w} \|_{2}] \leq \sqrt{4N-4} \cdot \zeta \,.
\end{equation}

The proof is completed by combining equations \eqref{2-2-1},
\eqref{2-2-2}, and \eqref{2-2-3}. 

\end{proof}


\subsection{Concentration of errors}\label{Concentration of errors}
The empirical performance of our algorithm is actually much better
than the theoretical bound given in Theorem \ref{proj}. This is
because the error term $\mathbf{u} - \mathbf{u}^{\text{ext}}$ usually
concentrates at the boundary of the domain. To see this, we can
calculate the basis of $\mathrm{Ker}(\mathbf{A})$. In 1D, $\mathrm{Ker}(\mathbf{A})$
only contains two linear functions, which has very little error
concentration. In 2D, for the case of the Laplacian, $f = 0$
on the unit square domain with $M = N$, a basis of $\mathrm{Ker}(\mathbf{A})$ can be explicitly
given. The following proposition follows easily from some elementary calculations.

\begin{pro}
\label{basis}
Let $\mathbf{A}$ be the discretized Laplacian on a square domain with
$N\times N$ grids. Without loss of generality assume $N$
is odd. Let $(x, y) = (\frac{i-1}{N-1}, \frac{j-1}{N-1})$ and $h =
1/(N-1)$. Define vectors $\mathbf{u}^{k, p} \in \mathbb{R}^{N\times
  N}, p = 1, \cdots, 8$ by
\begin{equation}
\label{laplace_basis}
\begin{split}
u^{k,1}_{i,j} &= b^{1}_k \sin(2 k \pi x) e^{c_{k}y}, \quad \quad
u^{k,2}_{i,j} = b^{2}_k \cos(2 k \pi x) e^{c_{k}y} \,\\
u^{k,3}_{i,j} &= b^{3}_k \sin(2 k \pi x) e^{c_{k}(1- y)},\quad
u^{k,4}_{i,j} = b^{4}_k \cos(2 k \pi x) e^{c_{k}(1-y)},\\
u^{k,5}_{i,j} &= b^{5}_k \sin(2 k \pi y) e^{c_{k}x}, \quad \quad
u^{k,6}_{i,j} = b^{6}_k \cos(2 k \pi y) e^{c_{k}x},\\
u^{k,7}_{i,j} &= b^{7}_k \sin(2 k \pi y) e^{c_{k}(1- x)},\quad
u^{k,8}_{i,j} = b^{8}_k \cos(2 k \pi y) e^{c_{k}(1-x)},
\end{split}
\end{equation}
where $c_{k} = h^{-1}cosh^{-1}(2 - cos(2 k \pi h))$, $k = 1, \cdots,
(N-3)/2$ for $u^{k,1}, u^{k,3}, u^{k, 5}, u^{k, 7}$ and $k = 1,
\cdots, (N-1)/2$ for $u^{k,2}, u^{k,4}, u^{k, 6}, u^{k, 8}$,
and $b^{p}_{k}$ are normalizers to make $\| \mathbf{u}^{k, p}\|_{2} = 1$. Further
define vectors $\mathbf{u}^{l, p} \in \mathbb{R}^{N\times N}$, $p = 1, 2, 3, 4$ such that
\begin{equation}
  \label{laplace_basis_2}
u^{l, 1}_{i,j} = 1, \quad
u^{l, 2}_{i,j} = x,\quad
u^{l, 3}_{i,j}  = y ,\quad
u^{l, 4}_{i,j}  = xy.
\end{equation}
Then 
$$
  \mathcal{B} = \{ \{\mathbf{u}^{k, p}\}_{p = 1}^{8}, \{\mathbf{u}^{l, p}\}_{p = 1}^{4}\}
$$
is a basis of
$\mathrm{Ker}(A)$. 
\end{pro}

\begin{figure}[!h]
        \includegraphics[width=0.6\linewidth]{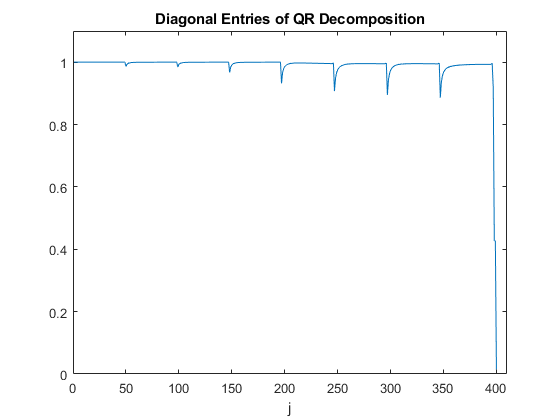}
\caption{Diagonal entries from the $R$ matrix of the $QR$ decomposition of the basis
$\mathcal{B}$ in~\eqref{laplace_basis}
and \eqref{laplace_basis_2} respectively arranged in lexocographical order. The basis
functions are nearly orthogonal. The mesh size $N$ is $101$. Value of last diagonal entry is
$0.015$. }
\label{orth_basis}
\end{figure} 
The basis $\mathcal{B}$ in Proposition \ref{basis} is nearly an orthonormal basis in
most directions. For example, a QR decomposition of vectors in $\mathcal{B}$ shows that most diagonal terms terms nearly equal to 1 even for large $N,$ as
shown in Figure~\ref{orth_basis}. The only exception is
$\mathbf{u}^{l, 4}$, whose corresponding diagonal term is only
$0.015$. The exponential terms in vectors $\mathbf{u}^{k, p}$ indicate
exponential decay of the solution away from the boundary. Since $4N-8$
out of $4N-4$ vectors in $\mathcal{B}$ has significant concentration
at the boundary, we expect the concentration of error at the boundary
with a high probability.

The basis of $\mathrm{Ker}( \mathbf{A})$ for the general case can not
be explicitly given. Instead, we can compute principal angles
between $\mathrm{Ker}(\mathbf{A})$ and $\Theta_{D}$, where $\Theta_{D}$ is the
subspace spanned by coordinate vectors corresponding to boundary
layer with thickness $D$. In other words,
$$
\Theta_{D} = \mathrm{span}\{ \mathbf{e}_{i,j} \,|\, i \leq D \mbox{ or
} j \leq D \mbox{ or } i \geq N-D \mbox{ or } j \geq N-D\} \,.
$$
If most principal angles are small, $\mathrm{Ker}(\mathbf{A})$ is almost
parallel with $\Theta_{D}$. And the projection of a random vector to
$\Theta_{D}$ preserves most of its length. In other words, we see a
concentration of error terms at the boundary of the domain. 

Principal angles $0 \leq \theta_{1} \leq \cdots \leq \Theta_{4N-4}
\leq \pi/2$ are a sequence of angles that describe the angle between
$\mathrm{Ker}(\mathbf{A})$ and $\Theta_{D}$. The first one is
$$
  \theta_{1} = \min\{ \arccos ( \frac{\bm{\alpha}\cdot \bm{\beta}}{\|
    \bm{\alpha}\| \| \bm{\beta} \|} ) \,|\, \bm{\alpha} \in
  \mathrm{Ker}(\mathbf{A}), \bm{\beta} \in \Theta_{D} \} = \angle
  (\bm{\alpha}_{1}, \bm{\beta}_{1})\,.
$$
Other angles are defined recursively with
$$
  \theta_{i} = \min\{ \arccos ( \frac{\bm{\alpha}\cdot \bm{\beta}}{\|
    \bm{\alpha}\| \| \bm{\beta} \|} ) \,|\, \bm{\alpha} \in
  \mathrm{Ker}(\mathbf{A}), \bm{\beta} \in \Theta_{D} , \bm{\alpha} \perp
  \bm{\alpha}_{j}, \bm{\beta} \perp \bm{\beta}_{j}, \forall 1\leq j \leq i-1\} \,,
$$
such that $\angle ( \bm{\alpha}_{i}, \bm{\beta}_{i}) =
\theta_{i}$. Without loss of generality assume $\|\bm{\alpha}_{i}\| =
1$ for all $1 \leq i \leq 4N-4$. Since
$\mathrm{dim}( \Theta_{D}) \geq \mathrm{dim}( \mathrm{Ker}(\mathbf{A}))$, it is easy to see that $\{ \bm{\alpha}_{1} , \cdots,\bm{\alpha}_{4N-4}\}$ forms an orthonormal basis of
$\mathrm{Ker}(A)$. Recall that $\mathbf{u} \in \mathrm{Ker}(\mathbf{A})$ and that the error
$\mathbf{u} - \mathbf{u}^{\text{ext}}$ is approximated by the project of a
random vector $\mathbf{w}$ with i.i.d. entries to the subspace $\mathrm{Ker}(\mathbf{A})$. Hence we
can further assume that $\bm{\xi} = \mathbf{u} - \mathbf{u}^{\text{ext}}$
is approximated by a random vector
\begin{equation}
\label{eq2-3-1}
  \bm{\xi} = \sum_{i = 1}^{4N-4} c_{i} \bm{\alpha}_{i} \,,
\end{equation}
where $c_{i}$ are i.i.d. random variables with
zero mean and variance $\zeta^{2}$. Define
$$
  p_{D}( \bm{\xi}) = \frac{\mathbb{E}[\| P_{\Theta_{D}}
    \bm{\xi}\|]}{\mathbb{E}[\| \bm{\xi}\|]}
$$
as the mean weight of $\mathbf{\xi}$ projected on to the boundary
layer, where $P_{\Theta_{D}}$ is the projection matrix to $\Theta_{D}$. Assume $\xi$ satisfies equation \eqref{eq2-3-1}. It is easy to see that
$$
  p_{D}( \bm{\xi}) = \frac{1}{4N-4}\sum_{i = 1}^{4N-4} \cos(
  \theta_{i}) \,.
$$
Therefore, $p_{D}( \bm{\xi})$ measures the degree
of concentration of errors on the boundary layer with thickness
$D$. 

Principal angles can be numerically computed by an SVD decomposition. In Figure
\ref{angles}, we list all principal angles for $D = 1, 2, 3$. The
matrices in Figure \ref{angles} are given by discretization of 2D
Fokker-Planck equations \eqref{F-P} 
for $f = 0$ (left panel) and $f$ as in equation \ref{ring} (right panel). The size of a
block is $(50 + D) \times (50 + D)$. We can see that the mean weight of $\bm{\xi}$ projected
to $\Theta_{D}$ is very large. In other words most of the error
term $\mathbf{u} - \mathbf{u}^{\text{ext}}$ concentrates at the
boundary layer. We also remark that the degree of concentration of
error terms increases with the dimension.

The 2D case is demonstrated in Figure \ref{angles}. And our
computation shows that the error concentration is even more significant in 3D.

\begin{figure}[!h]
        \includegraphics[width=\linewidth]{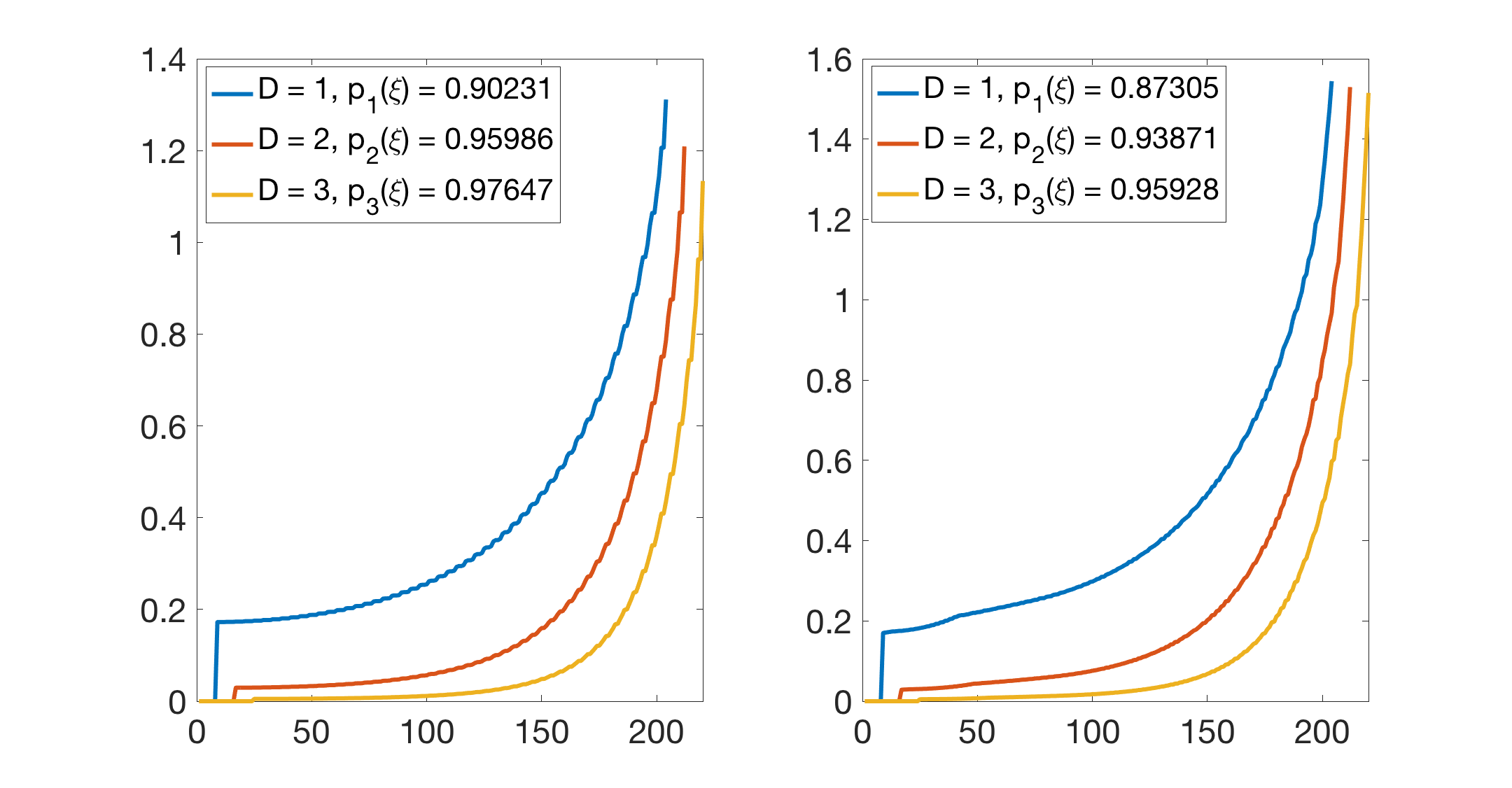}
\caption{Principal angles between $\mathrm{Ker}(\mathbf{A})$ and $\Theta_{D}$
  for $D = 1, 2, 3$. Left: Diffusion process without drift. Right:
  Fokker-Planck equation as given in Example \ref{Ring density function}. }
\label{angles}
\end{figure} 

Figure \ref{blockerror} shows an empirical test of the spatial
distribution of error terms. The Fokker-Planck equation is still from
the ring density function as in Section \ref{Ring density function}. We choose a $64\times 64$ block on $[0, 1]\times [0, 1]$
and solve the Fokker-Planck equation with our hybrid solver. The
Monte Carlo simulation uses $10^{6}$ sample points. The numerical
solutions $\mathbf{v}$ and $\mathbf{u}$ are compared with the exact
solution $\mathbf{u}^{\text{ext}}$ in the top left and right panel,
respectively. As a comparison, we also produce $10^{6}$ unbiased samples from the
invariant density itself, denoted by $\mathbf{v}'$. The solution of
the hybrid solver from $\mathbf{v}'$ is denoted by $\mathbf{u}'$. We can clearly
see that most error of $\mathbf{u}$ and $\mathbf{u}'$ concentrates at the boundary of
the domain. The bottom panel compares the relative weight of error
concentrating on the boundary layer for $D = 1, 2, 3, 4$. The relative
weights $\rho_{u}$ and $\rho_{v}$ are given by
$$
  \rho_{v} = \frac{\| P_{\Theta_{D}} ( \mathbf{v} - \mathbf{u}^{\text{ext}}) \|}{\|
    \mathbf{v} - \mathbf{u}^{\text{ext}} \|} \quad , \mbox{and } \quad \rho_{u} = \frac{\| P_{\Theta_{D}} ( \mathbf{u} - \mathbf{u}^{\text{ext}}) \|}{\|
    \mathbf{u} - \mathbf{u}^{\text{ext}} \|} \,,
$$
respectively. $\rho_{u'}$ and $\rho_{v'}$ are also defined
analogously. 

From Figure \ref{blockerror}, the spatial concentration of $\mathbf{u} -
\mathbf{u}^{\text{ext}}$ on the boundary layer is less than the theoretical prediction given
before, mainly because the sample itself has bias. But we can still see a significant concentration of error on
the boundary layer. The error concentration of $\mathbf{u}' -
\mathbf{u}^{\text{ext}}$ is much better. Almost all errors of
$\mathbf{u}'$ are concentrated on the two boundary layers. It is worth to mention that although the unbiased sample $\mathbf{v}'$ has
little visual difference from the Monte Carlo data $\mathbf{v}$, the
resultant solution $\mathbf{u}'$ has significant better performance in
terms of error concentration on the boundary. Hence this example
also demonstrates the importance of choosing a good Monte Carlo sampler. 

\begin{figure}[htbp]
\centerline{\includegraphics[width = \linewidth]{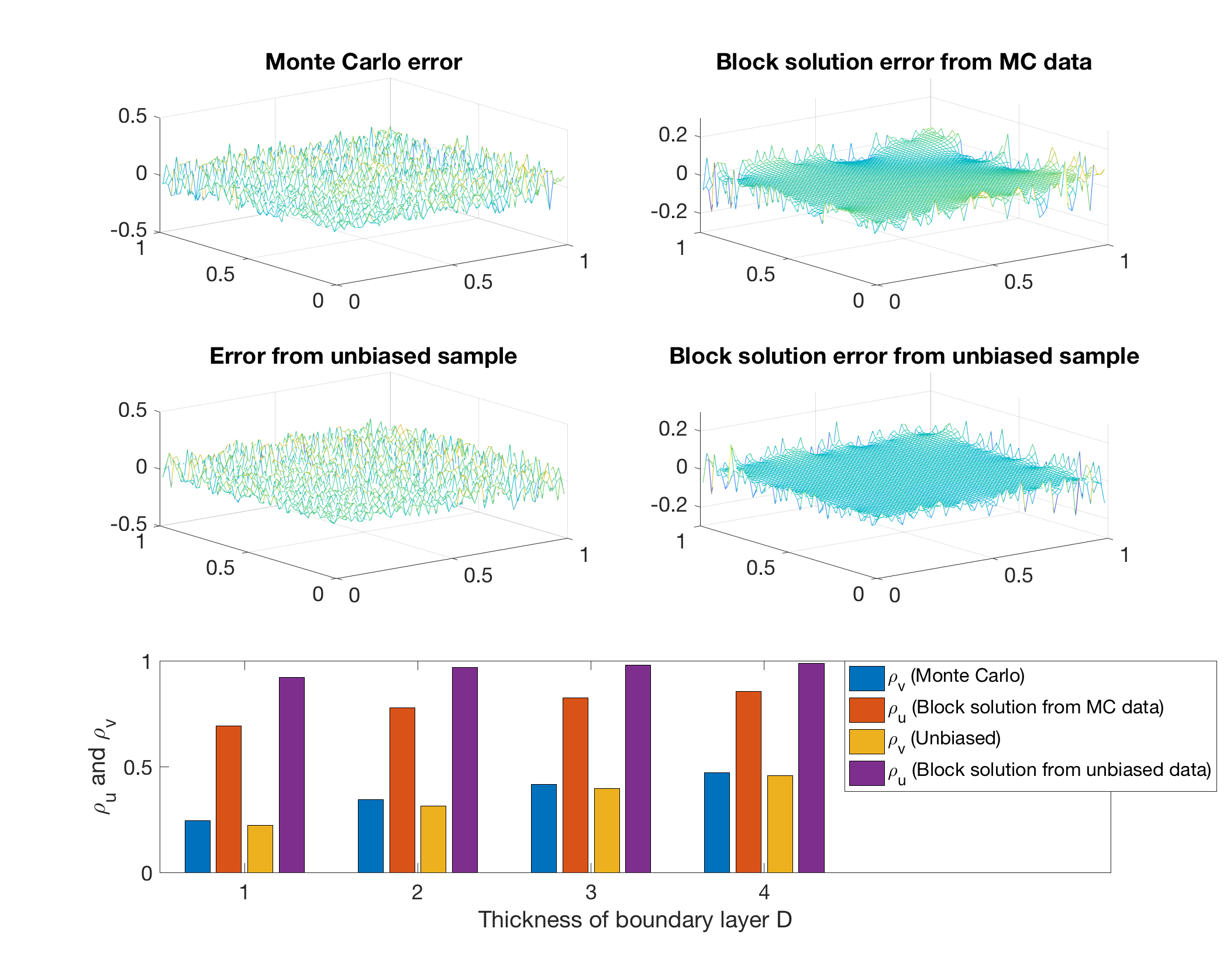}}
\caption{Empirical spatial distribution of error term for the ring
  density function as in Section \ref{Ring density function}.
  Top left: $\mathbf{v} - \mathbf{u}^{\text{ext}}$. Top right: $\mathbf{u} -
  \mathbf{u}^{\text{ext}}$. Middle left: $\mathbf{v}' -
  \mathbf{u}^{\text{ext}}$. Middle right: $\mathbf{u}' -
  \mathbf{u}^{\text{ext}}$. Bottom: Comparison of $\rho_{v}$,
  $\rho_{u}$, $\rho_{v'}$, $\rho_{u'}$ for $D = 1, 2, 3, 4$.}
\label{blockerror}
\end{figure}


\section{Block Fokker-Planck Solver}\label{Block Fokker-Planck Solver}


Since we can use the hybrid method to compute the Fokker-Planck equation on any
region in the phase space, a straightforward improvement is to apply
the divide-and-conquer strategy. We can
divide the interested numerical domain into small blocks and then
combine the results on these blocks to generate the solution on the
original big domain. As discussed in the introduction, assume we divide
an $N^{d}$ mesh into many $m^{d}$ blocks, where $m \ll N$. If the linear
solver to an $n\times n$ matrix has $O(n^{p})$ complexity (usually $p
> 1$), the total computational cost is reduced from $N^{pd}$ to
$m^{(p-1)d}N^{d}$. In addition, this block solver significantly simplifies parallel
computing, since all blocks are independent and satisfy the same
Fokker-Planck equation. We can also change the grid size for each
block based on whether the data is dense or sparse in a subregion to
further reduce the computational cost. Moreover, we can apply our method to
problems with irregular domains by dividing it into many small
rectangular blocks. 

For simplicity, we still use a rectangular domain $D = [a, a'] \times
[b, b']$ to describe our algorithm, and assume that we want to solve
$u$ in $D$. We divide $D$ into $K\times L$ blocks $\{D_{k,l}\}_{k=1,
  l=1}^{k=K, l=L}$ with $D_{k,l} = [a_{k-1}, a_{k}] \times [b_{l-1},
b_{l}]$, where $a_{k} = a+k(a'-a)/K$ and $b_{l} = b+l(b'-b)/L$.  

Following the algorithm presented in Section \ref{Algorithm description}, we
construct an $N\times M$ grid on $D_{k,l}$ and discretize the
Fokker-Planck equation. This gives a linear constraint
$$
  \mathbf{A_{k,l}} \mathbf{u^{k,l}} = \mathbf{0}
$$
on $D_{k,l}$, where $\mathbf{A_{k,l}}$ is a $(N-2)(M-2) \times (NM)$
matrix. Then we obtain a reference solution $\mathbf{v^{k,l}}$ from
the Monte-Carlo simulation by picking up the corresponding values
$\{v^{k,l}_{i,j}\}_{i = 1, j = 1}^{i = N, j = M}$ from the global
simulation result $\mathbf{v}$, such that   
$$v^{k,l}_{i,j}=v((k-1)N+i,(l-1)M+j) $$
for $k=1,\dots,K, l=1,\dots,L, i=1,\dots,N, j=1,\dots,M$. This gives
an optimization problem on $D_{k, l}$
\begin{eqnarray}
\label{blockopt}
 \mbox{min} &  & \| \mathbf{u} - \mathbf{v_{k,l}} \|_{2} \\\nonumber
\mbox{subject to } & & \mathbf{A_{k,l}} \mathbf{u} = \mathbf{0}  \,.
\end{eqnarray}
We denote the solution to \eqref{blockopt} by $\mathbf{u_{k,l}}$,
which can be obtained by calculating
$$
  \mathbf{u_{k,l}} = \mathbf{A_{k,l}}^{T}( \mathbf{A_{k,l}}
  \mathbf{A_{k,l}}^{T})^{-1}(-\mathbf{A_{k,l}} \mathbf{v_{k,l}})+\mathbf{v_{k,l}} \,.
$$ 

Now, the $(i,j)$ coordinate $u^{k,l}_{i,j}$ of $\mathbf{u_{k,l}}$ is
an approximation of $u$ at the point $(ih + a_{k-1} - h/2, jh +
b_{l-1} - h/2)$, where $h=(a_{k} - a_{k-1})/N = (b_{l} - b_{l-1})/M$
is the grid size when we divide $D_{k,l}$ into $N\times M$
boxes. It remains to combine all local solutions $\{\mathbf{u_{k,l}}\}_{k = 1, l =
  1}^{k = K, l = L}$ on all blocks by collaging them together, i.e., 
$$
  u(ih + a_{k-1} - h/2, jh + b_{l-1} - h/2) = u^{k,l}_{i,j} \,,
$$ 
$k=1,\dots,K, l=1,\dots,L,
i=1,\dots,N, j=1,\dots,M$. The collage numerically solves the
Fokker-Planck equation \eqref{F-P} on the whole domain $D$. 


\section{Reducing Interface Error}\label{Reducing Interface Error}

As discussed in Section \ref{Concentration of errors}, the
optimization problem \eqref{opt} projects most error terms to the
boundary of the domain. For the block algorithm, the solution is less 
accurate near
the boundary of each block. The error on the boundary usually looks
noisy because it inherits the randomness from Monte Carlo
simulations. As a result, there are visible fluctuations on the
interface of two adjacent blocks. To make the block solver applicable,
modifications to the solution on the interface of blocks are necessary. 

In this section, we provide two
different methods to reduce the interface
error, i.e., the overlapping blocks method and the shifting blocks
method. The overlapping blocks method expands each block locally, and keeps only the interior portion which has much lower observed errors.  The shifting blocks method makes several smoothing passes, shifting the block boundaries each time so that portions previously
on the edges are now in block interiors. Advantages and limitations of these methods will also be
discussed. 


\subsection{Overlapping blocks}\label{Overlapping blocks}

Since the numerical solution of the hybrid solver has much higher
accuracy at interior points than on the boundary, the most natural
approach is to discard the boundary layer. When applying the block solver, we can enlarge the
blocks by one or two layers of boxes. Then we apply the algorithm in
Section \ref{Algorithm description} on the enlarged block. The
interior solution restricted to the original block is the new output of the block
solver. This is called the {\it overlapping blocks} method.

More precisely, recall that we first divide $D = [a, a'] \times [b,
b']$ into $K\times L$ blocks $\{D_{k,l}\}_{k=1, l=1}^{k=K, l=L}$, then
divide each block $D_{k,l} = [a_{k-1}, a_{k}] \times [b_{l-1},b_{l}]$
into $N \times M$ boxes $O^{k,l}_{i,j} = [a_{k-1} + (i-1)h, a_{k-1}
+ih] \times [b_{l-1} + (j-1)h, b_{l-1} + jh]$, where $h = (a_{k} -
a_{k-1})/N = (b_{l} - a_{l-1})/M$. Instead of $D$, now we work on the extended domain
$\tilde{D}_{k,l} = [a_{k-1}-\iota h, a_{k}+\iota h] \times
[b_{l-1}-\iota h,b_{l}+\iota h]$, where $\iota=1$ or $2$. Then the
Monte-Carlo simulation is used to get the reference solution $\mathbf{\tilde{v}}$ on the
enlarged domain $\tilde{D} = [a-\iota h, a'+\iota h] \times [b-\iota h,
b'+\iota h]$ of $D$.  

Instead of disjoint blocks $D_{k,l}$, we construct an $(N + 2 \iota) \times (M + 2 \iota)$ grid on $\tilde{D}_{k,l}$ and generate the discretized Fokker-Planck equation
$$
  \mathbf{\tilde{A}_{k,l}} \mathbf{u} = \mathbf{0}
$$
on $\tilde{D}_{k,l}$, where $\mathbf{\tilde{A}_{k,l}}$ is a
$(N+2\iota-2)(M+2\iota-2) \times ((N+2\iota)(M+2\iota))$ matrix. Then
a local reference solution $\mathbf{\tilde{v}^{k,l}}$ is obtained by picking
up the corresponding value $\tilde{v}^{k,l}_{i,j}$ from the global
simulation vector $\mathbf{\tilde{v}}.$
For each block,
we solve the local optimization problem~\eqref{blockopt}, and keep only
the values at the interior points, $D_{k,l}$ to create the global approximation $u.$ 

The advantage of this overlapping block method is that it is very easy
to implement. No additional treatment is necessary besides discarding
one or two boundary layers. But in higher dimension, a significant
proportion of grid points will be on the boundary of blocks. For
example, if $\iota = 2$, $M = N = 30$, the percentage of unused grid
points is $13 \%$ in 1D, $28 \%$ in 2D, $46 \%$ in 3D, and $65 \%$ in
4D. Also, as seen in Figure \ref{blockerror}, visible error inherited
from the reference solution $\mathbf{v}$ can easily penetrate through
$4-5$ boundary layers. Hence the output of solutions from the
overlapping block method usually still have some visible residual interface
error. 

\subsection{Shifting blocks}\label{Shifting blocks}

The idea of {\it shifting block} is also motivated by the concentration of
error of the solution of \eqref{opt}. To resolve the interface
fluctuation between blocks, one can simply move the interface to the
interior by shifting all blocks and recalculate the solution. Since
the solution has much higher accuracy in the interior of a block, this can
easily smooth the interface error. More precisely, after applying the
block solver, we make a ``half-block" shift
of the blocks so that boundaries of the original
blocks are now in the interior of new blocks. Then we solve
optimization problems \eqref{opt} again on newly shifted
blocks. The reference solution fed into the optimization problem
\eqref{opt} is the numerical solution from the first round. If
necessary, one can carry out this shifting block for several rounds to
cover all grid points and to improve the accuracy.

Divide the domain $D = [a, a'] \times [b, b']$ into $K\times
L$ blocks $\{D_{k,l}\}_{k=1, l=1}^{k=K, l=L}$ with $D_{k,l} =
[a_{k-1}, a_{k}] \times [b_{l-1}, b_{l}]$, where $a_{k} = a+k(a'-a)/K$
and $b_{l} = b+l(b'-b)/L$. Then we make half-block shifts to get the
shifted blocks $D'_{k,l} = [a'_{k-1}, a'_{k}] \times [b'_{l-1},
b'_{l}]$, where $a'_{k} = a_{k} +(a'-a)/2K= a+(k+1/2)(a'-a)/K$ and
$b'_{l} =b_{l} +(b'-b)/2L = b+(l+1/2)(b'-b)/L$. The Monte Carlo data
needs to cover all blocks $D_{k,l}$ and $D'_{k,l}$. 

Now construct $N\times M$ grids both on $D_{k,l}$ and $D'_{k,l}$, and generate the discretized Fokker-Planck equations
$$
  \mathbf{A_{k,l}} \mathbf{u} = \mathbf{0}\qquad
  \text{and}\qquad
  \mathbf{A'_{k,l}} \mathbf{u} = \mathbf{0}
$$
on $D_{k,l}$ and $D'_{k,l}$ respectively, where $\mathbf{A_{k,l}}$ and
$\mathbf{A'_{k,l}}$ are $(N-2)(M-2) \times (NM)$ matrices.  

We first use the original block solver to solve the optimization problem
on each $D_{k,l}$, as described in Section \ref{Block Fokker-Planck
  Solver}. This gives approximated solutions $\mathbf{u}_{k,l}$ on
each block. The first approximation $\mathbf{u}^{1} \in
\mathbb{R}^{MN}$ is obtained by collaging $\mathbf{u}_{k,l}$ from all blocks. 

Then we generate the reference solution $\mathbf{v'^{k,l}}$ on shifted
blocks $D'_{k,l}$ by using the corresponding values in
$\mathbf{u}^{1}$ whenever available. More precisely we have
$$\mathbf{v}'^{k,l}_{i,j}=\mathbf{u}^{1}_{(k-1/2)N+i,(l-1/2)M+j} $$
for $k=1,\dots,K-1, l=1,\dots,L-1, i=1,\dots,N, j=1,\dots,M$. When $k
= K$ or $l = L$, we use Monte Carlo data to produce
$\mathbf{v}'^{k,l}_{i,j}$ if $\mathbf{u}_{1}$ data is not available. Then we solve
the optimization problem~\eqref{blockopt}
on the shifted block $D'_{k,l}$ to get a numerical solution
$\mathbf{u'_{k,l}}$ on $D'_{k,l}$. 


Now $\mathbf{u}'_{k,l}$ are computed on shifted blocks. We use data from
$\mathbf{u}'_{k,l}$ to produce the global solution whenever possible, that is, let 
$$
  u((k-1)N + ih + a_{k-1} - h/2, (l-1)M
+ jh + b_{l-1} - h/2) = u'^{k,l}_{i,j}
$$
for $k=2,\dots,K, l=2,\dots,L, i=1,\dots,N, j=1,\dots,M$. If $k = 1$
or $l = 1$, we use values from $\mathbf{u}^{1}$ if the data from
$\mathbf{u}'_{k,l}$ is not available. 

We remark that in practice one does not have to shift the block by
exactly one half. This shifting block method can be implemented
repeatedly, such that the solution $\mathbf{u}'$ from last round is used as the
reference solution for the next round. We find that one efficient way
of implementation is to shift the block by $1/3$ for two times to get
two solutions $\mathbf{u}'$ and $\mathbf{u}"$ on shifted blocks. Then
we feed $\mathbf{u}"$ back to the original block solver as the
reference solution. This implementation covers all grid points by
interiors of blocks. Using an iterative linear solver can
significantly accelerate the shifting blocks method. Because from the
second round, we have $\mathbf{u}_{k,l} \approx \mathbf{v}_{k, l}$ at
all interior grid points. Hence $\mathbf{0}$ is a good
initial guess when solving $(\mathbf{A}_{k,l}
\mathbf{A}_{k,l}^{T})^{-1}(-\mathbf{A}_{k,l} \mathbf{v}_{k,l})$ in the
optimization problem \eqref{blockopt}. Empirically, the total
computation time of three shifts is roughly similar to the time needed
for the first round, if the conjugated gradient linear solver is used.

\section{Numerical Examples}\label{Numerical Examples}
In this section, we consider the following
three numerical examples to test the performance of our methods.  

\begin{figure}
	\center{
                \includegraphics[width=0.49\textwidth]
		{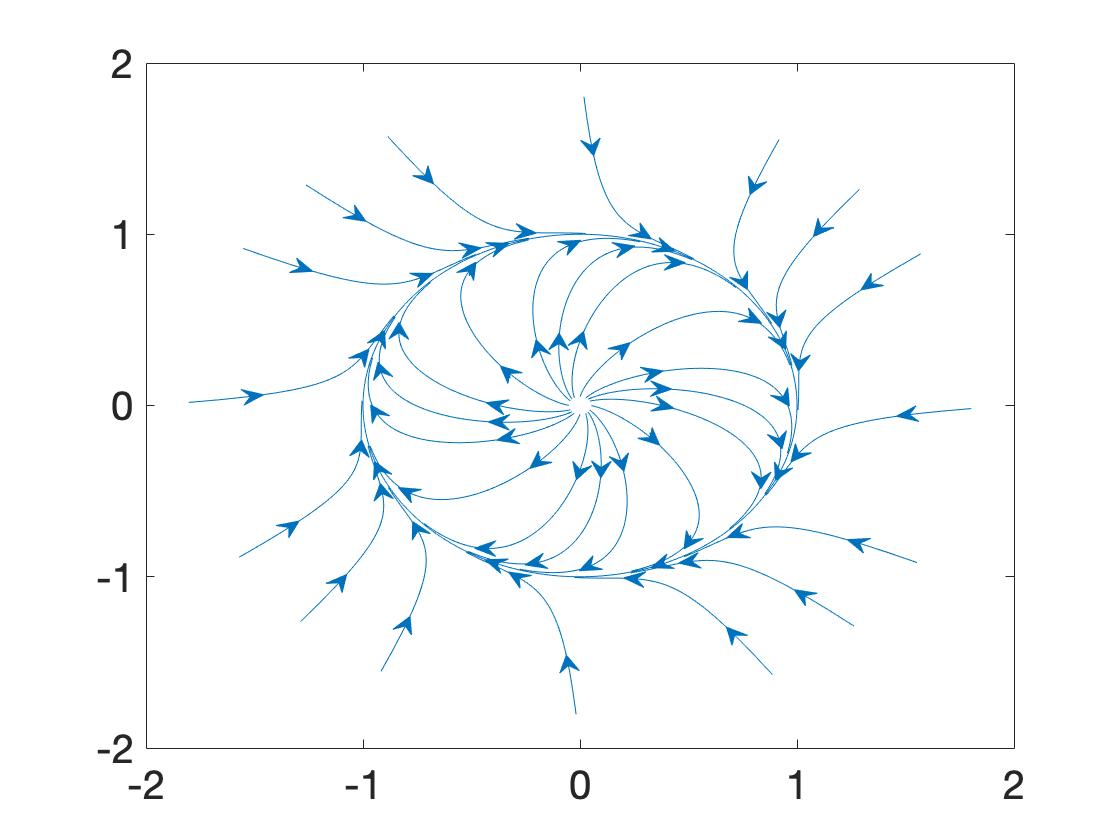}
		\includegraphics[width=0.49\textwidth]
		{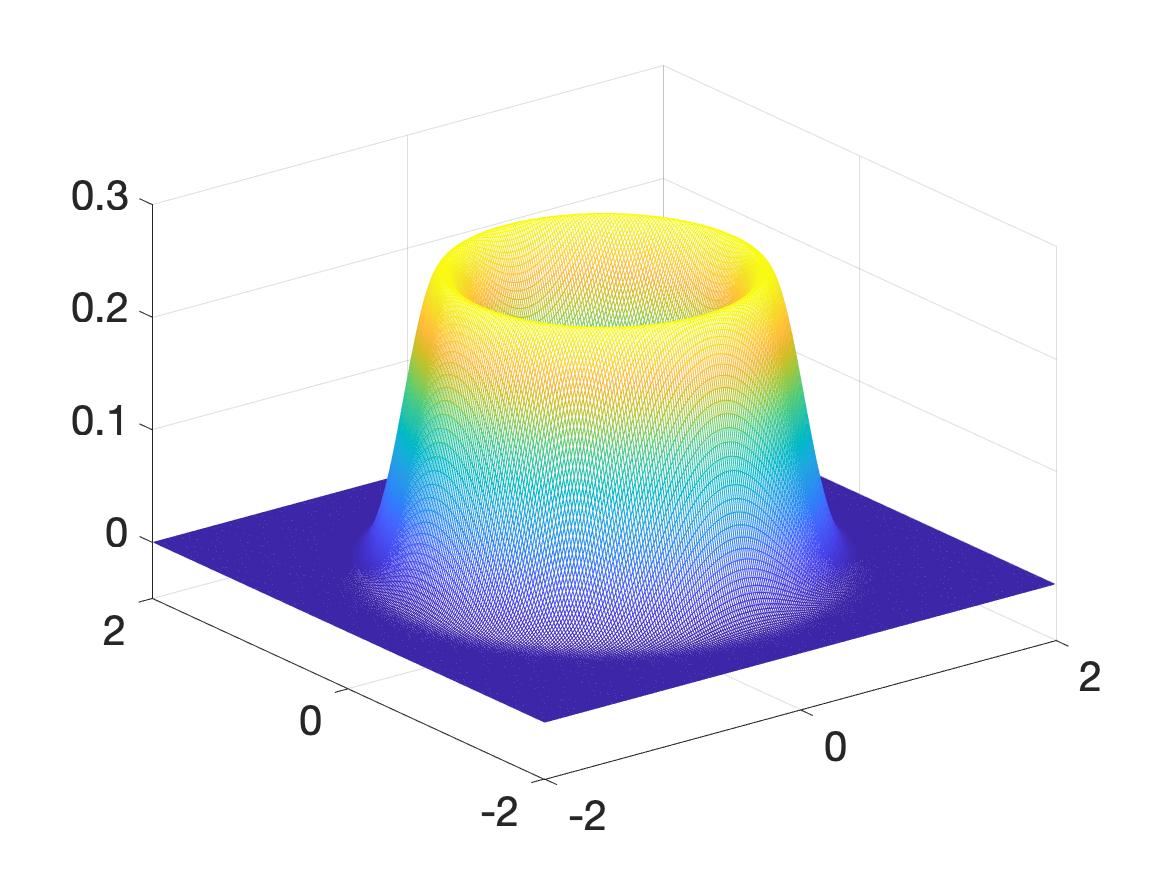}
	}
	\caption{Left: Some trajectories of the deterministic part of
          equation \eqref{ring}. Right: Exact solution of the Fokker-Planck equation
          for \eqref{ring}. }
\label{ring_exact}
\end{figure}

\subsection{Ring density function}\label{Ring density function}
Consider the following stochastic differential equation:
\begin{equation}\label{ring}
\left\{
\begin{array}{l}
dx=\big(-4x(x^2 + y^2 - 1) + y\big)\,dt + \varepsilon\,dW_t^x\\
dy=\big(-4y(x^2 + y^2 - 1) - x\big)\,dt + \varepsilon\,dW_t^y
\end{array}
\right.,
\end{equation}
where $W_t^x$ and $W_t^y$ are independent Wiener processes. To compare
the performance of different solvers in this paper, we fix the
strength of white noise to be $\varepsilon=1$. The deterministic part
of equation \eqref{ring} is a gradient system plus a perpendicular
rotation term, where the potential function of the gradient component is
$$V(x,y)=(x^2+y^2-1)^2.$$ 
See Figure \ref{ring_exact} Left for selected trajectories of equation
\eqref{ring}. The rotation term does not change the invariant probability density
function. Therefore, the deterministic part of equation \eqref{ring}
admits a limit circle $x^2+y^2=1$, and the invariant probability measure
of \eqref{ring} has density function 
$$u(x,y)=\frac{1}{K}e^{-2V/\varepsilon^2},$$
where $K=\pi\int_{-1}^{\infty}e^{-2t^2/\varepsilon^2}\,dt$ is the
renormalization constant. Therefore, the stationary Fokker-Planck equation
corresponding to \eqref{ring} has an analytic global solution $u(x,y)$ on
$\mathbb{R}^2$ (Figure \ref{ring_exact} Right). 

\begin{figure}
	\center{
		\includegraphics[width=0.49\textwidth]
		{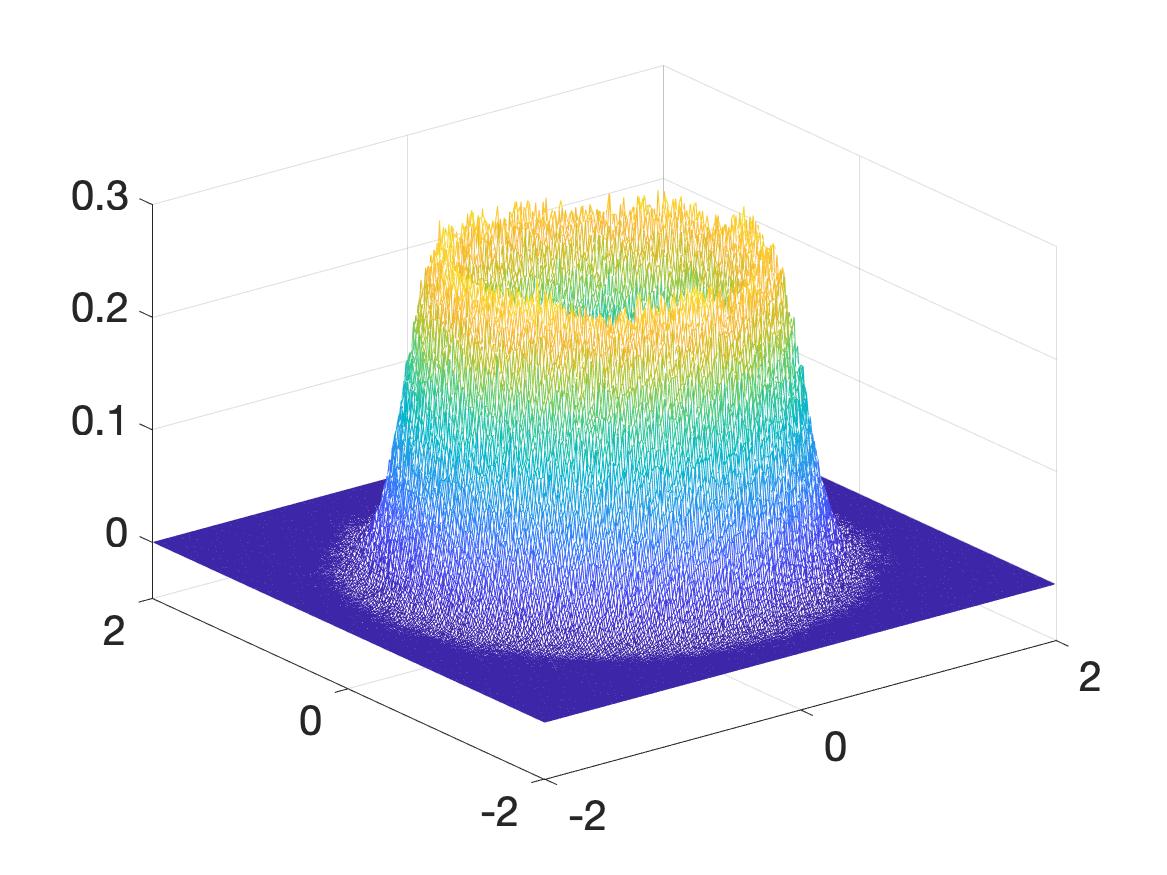}
		\includegraphics[width=0.49\textwidth]
		{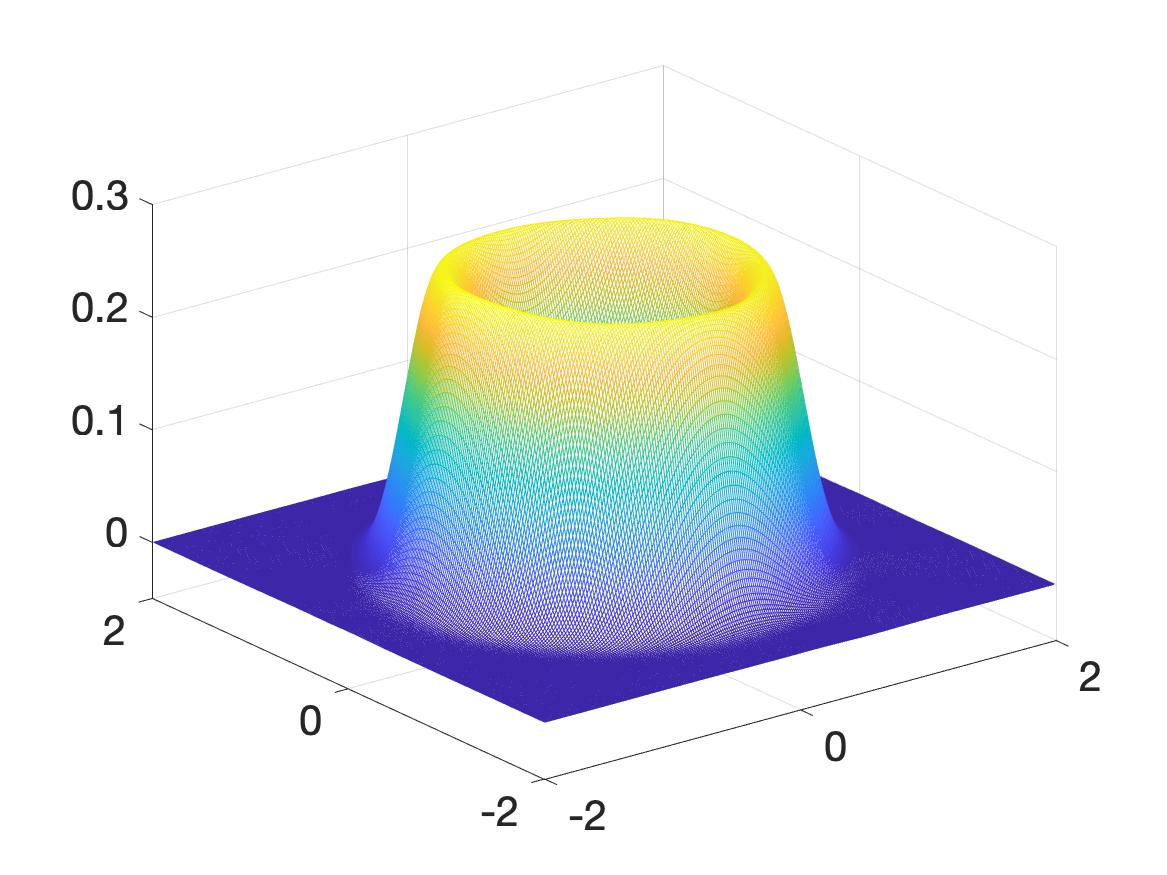}
	}
	\caption{\textbf{(Ring density)} The approximation by Monte Carlo simulation (\textbf{left}) and the algorithm in Section \ref{Algorithm description} (\textbf{right}) with $256\times256$ mesh points and $10^7$ samples.}\label{ring_data}
\end{figure}

We first look at the approximation obtained from Monte-Carlo
simulation with $256\times256$ mesh points on the domain
$D=[-2,2]\times[-2,2]$, and use step size $dt=0.002$ and $10^7$
samples in the sampling step (the left figure in Figure
\ref{ring_data}). As expected, we can see that this approximation has
too much fluctuation to be an acceptable solution of the stationary Fokker-Plank equation. But
the algorithm in Section \ref{Algorithm description} provides a
smoothed approximation of the exact solution (the right figure
in Figure \ref{ring_data}).   

\begin{figure}
	\center{
		\includegraphics[width=0.49\textwidth]
		{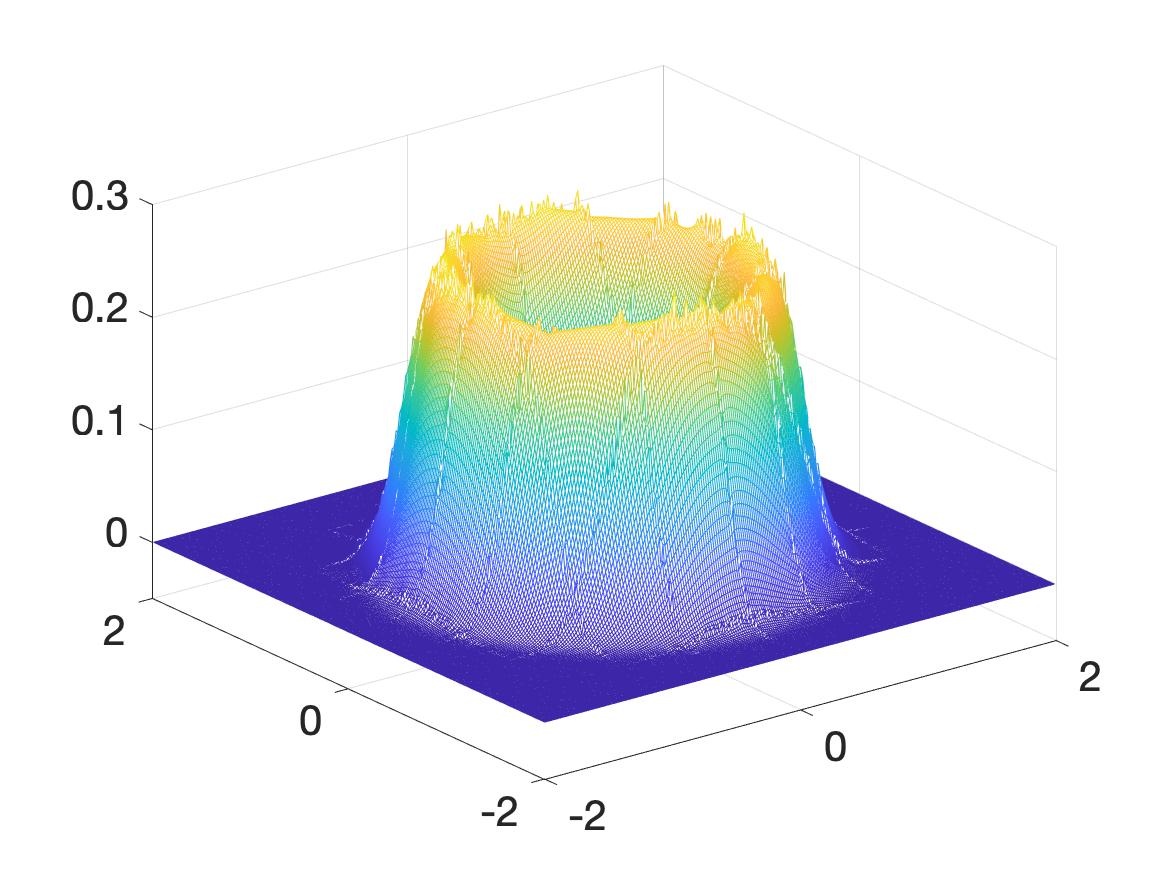}
		\includegraphics[width=0.49\textwidth]
		{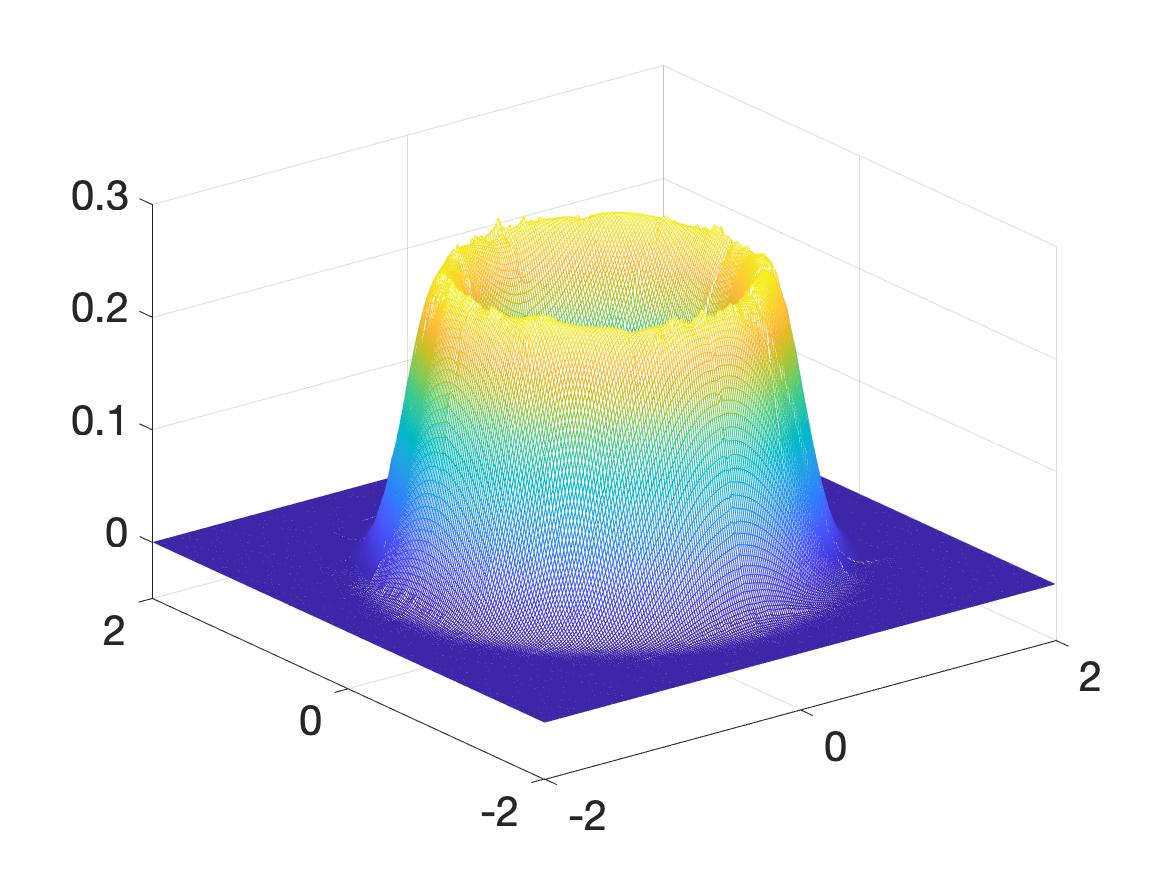}
	}
	\caption{\textbf{(Ring density)} The approximation computed by the basic block solver (\textbf{left}) and $1$-overlapping block solver (\textbf{right}) with $256\times256$ mesh points, $8\times8$ blocks and $10^7$ samples.}\label{ring_block}
\end{figure}

\begin{figure}
	\center{
		\includegraphics[width=0.55\textwidth]
		{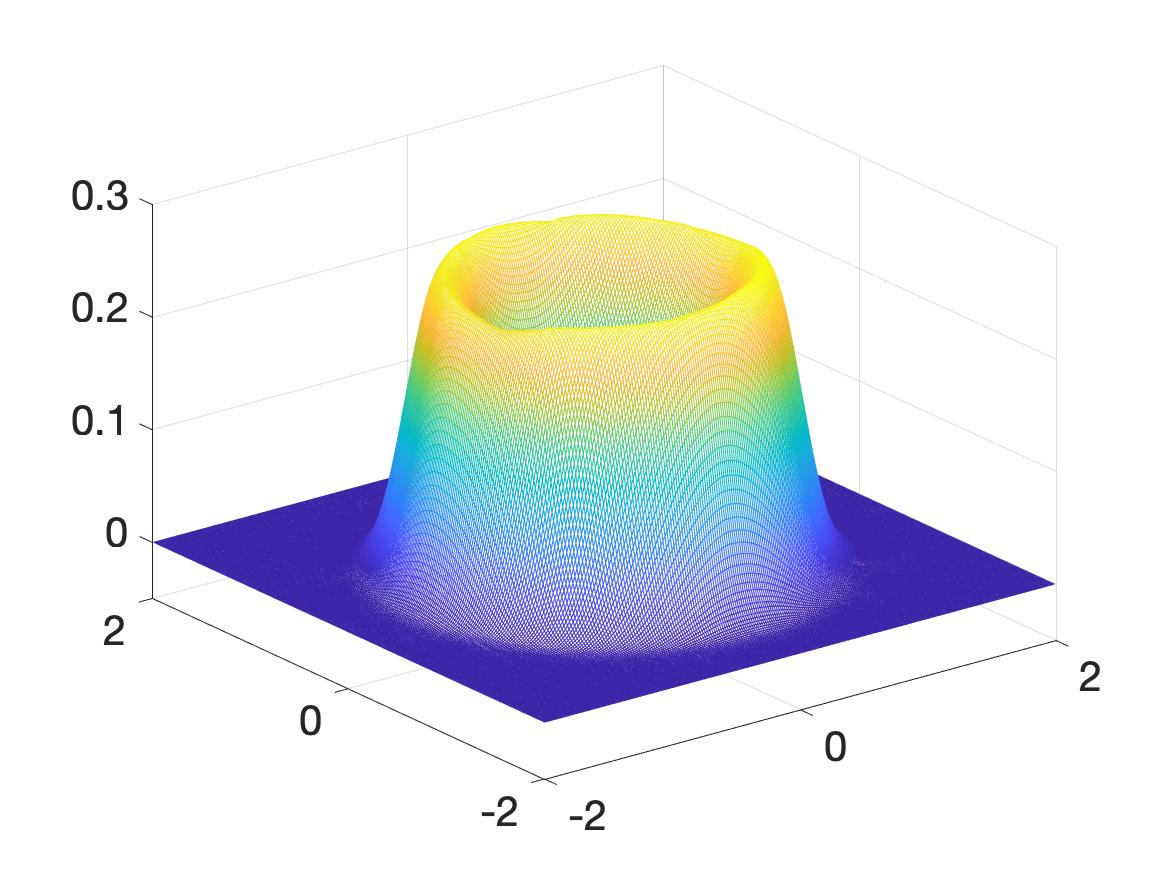}
	}
	\caption{\textbf{(Ring density)} The approximation computed by a triple iterated half-block shifting solver (\textbf{left}) with $256\times256$ mesh points, $8\times8$ blocks and $10^7$ samples.}\label{ring_HS_RG}
\end{figure}

Then we test the performance of block solvers, which is the theme of
the present paper. In addition, we need to compare the effect of two
error reduction methods proposed in Section \ref{Reducing Interface
  Error}. In the next a few figures, we still use a $256\times256$ mesh on the
domain $D=[-2,2]\times[-2,2]$, and $10^7$ samples to simulate the
reference data $\mathbf{v}$. We further divide $D$ into $8\times8$ blocks, each of which thus has $32\times32$ mesh points. The left figure
in Figure \ref{ring_block} is the approximation given by the naive
block solver described in Section \ref{Block Fokker-Planck Solver}. As
expected in Section \ref{Concentration of errors} and explained at the
beginning of Section \ref{Reducing Interface Error}, the error term of
the Monte-Carlo simulation data $\mathbf{v}$ (see Figure \ref{ring_data} 
(left)) is spread from the interior of each block to its boundary
because of the projection, which causes visible interface fluctuation.  

The next step is to implement two different error reduction methods introduced in Section \ref{Reducing Interface
  Error} and compare their performances. The right of Figure \ref{ring_block} shows the solution
given by overlapping blocks with $1$ layer of box overlap, that is,
$\iota=1$ (see Section \ref{Overlapping blocks}). We can see that the
interface fluctuation is reduced, especially at the places with high
probability density function and high interface fluctuation. Figure \ref{ring_HS_RG} is obtained 
by iterating the shifting block solver (see Section \ref{Shifting
  blocks}) for three repeats. The interface fluctuation is not only reduced,
but also smoothed significantly. 

Here we compare numerical solutions of the invariant probability
measure of equation \eqref{ring}, which is explicitly known. Figure \ref{Convergence}
shows a comparison of error terms for solutions obtained by different
error reduction methods, in both discrete $L^2(D)$
norm and discrete $H^{1}(D)$ norm. To make
a fair comparison, we let the number of samples change with the grid
size. Examples with mesh sizes $N = 64, 128, 256, 512, 1024$, and $2048$ are
tested and compared. The block size is $32\times 32$ in all tests. The
total number of Monte Carlo samples is chosen to be $390.625
N^{2}$. From Figure \ref{Convergence}, we can see that the $L^{2}$
error of the Monte Carlo data is stabilized as expected, because the
average sample count per box (and per grid) is constant.

The performance of two error reduction methods are compared in Figure
\ref{Convergence}. We can see that the plain block solver reduces the
error significantly compared to the Markov chain data, but the error does 
not seem to converge to zero. This
is not a surprise because all blocks are $32\times 32$, and Theorem
\ref{proj} says that the error should be proportional to $-1/2$ power
of the block size. Both error reduction methods reduces the
$L^{2}$ error from the plain block solver to some degree. The shifting
blocks method has better performance, but also a higher computational cost. We can see that the empirical rate of error decay for
  the shifting block method is roughly $N^{-1/2}$, which is better than the theoretical result in
  Theorem \ref{proj}. 

 First order derivatives in the discrete $H^{1}(D)$ norm are calculated
 by taking finite differences with respect to nearest grid
 points. With a constant mean sample size per box, the $H^{1}(D)$
 error of the reference solution from Monte Carlo data diverges when $N$ increases. This is because local fluctuations are roughly
  unchanged with the mesh size, while the grid size $h$ become
  smaller. Therefore, the derivative of the reference solution is
  $O(\zeta h^{-1})$, where $\zeta$ is the standard deviation of number
  of samples per box. In other words, all algorithms based on Monte
  Carlo simulations are expected to have poor performance in $H^{1}(D)$
  error. The divergence of $H^{1}(D)$ error is alleviated
  by the overlapping block method, and partial overturned
  by the shifting blocks method. As we see in Figure \ref{Convergence}
  Right, when $N = 2048$, the shifting block
  method gives a solution whose $H^{1}(D)$ error is $> 150$
  times less than that of the Monte Carlo data.

We can see that due
to the lack of interaction between blocks, the information of the
reference solution obtained by the Monte-Carlo simulation is not transferred to
a neighboring blocks.  So if a block is over-sampled, while the others
are under-sampled, then after the block solutions are pasted together,
the graph is not ``flat" at the places where it should be. We can see
that the shifting block method has better performance in terms of improving the
regularity. This is because it significantly increases interactions
between the neighbourhood blocks, and transfers the information between
neighborhood blocks. Applying the shifting block method repeatedly can
make the result more close to the global solver
or the exact solution. But it also incurs some extra computational
cost, as seen in Table \ref{cputime}.

\begin{figure}
	\center{
		\includegraphics[width=1.0\textwidth]
		{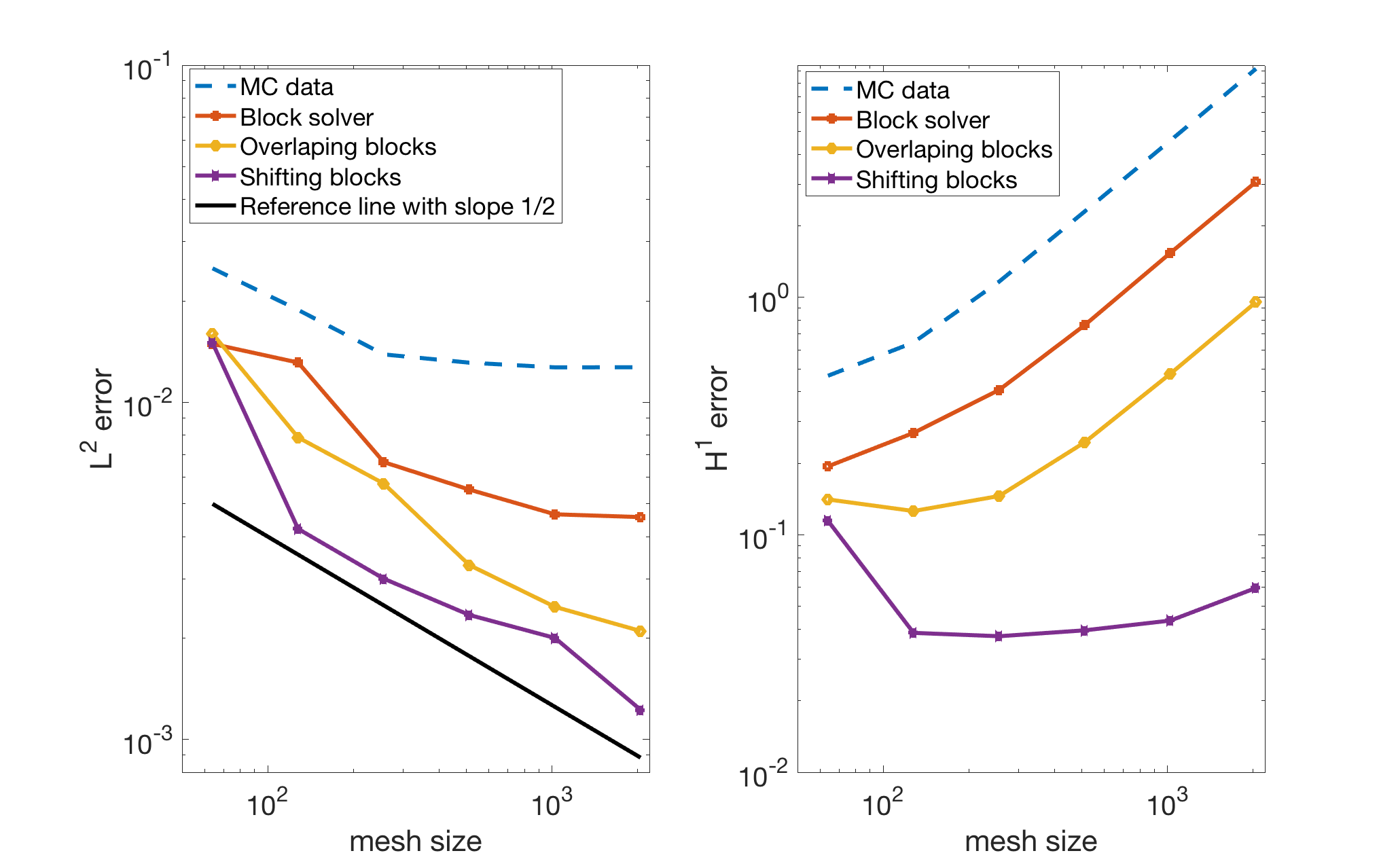}
	}
	\caption{\textbf{(Ring density)} Left: Discrete
          $L^2(D)$ error of solutions produced by different
          method. Right: Discrete $H^{1}$ error of solutions produced
          by different method.}\label{Convergence} 
\end{figure}

Finally, we show a comparison of computation time in Table
\ref{cputime}. In Table \ref{cputime}, ``Sampling'' means the Monte Carlo sampling time
(including a burn-in time, which is the waiting time before collecting
samples). ``Plain'' means the plain block solver proposed in Section 3. ``Overlapping'' means overlapping blocks method with $\iota
  = 1$ in Section 4.1. ``Shifting'' means the shifting blocks method in
  Section 4.2. We shift blocks twice by $1/3$ and $2/3$, and feed the
  new solution to the original solver as the reference
  solution.  And ``Old Version'' means the algorithm proposed in
  \cite{li2018data}, where no block is used. To make a fair comparison, no
  parallelizations or iterative linear solvers are used in this
  performance testing. We can see that the Monte Carlo sampling
  actually takes most of the time, and all versions of block-based
  solvers are very fast. When the mesh size is $2048$, the plain block
  solver is $>100$ times faster than solving a large optimization
  problem \eqref{opt} without dividing the domain.

\begin{table}[H]
\begin{tabular}[tb]{|c|c|c|c|c|c|c|}
\hline
Mesh&Sampling&Plain&Overlapping&Shifting& Old Version\\
\hline
64&0.5697&0.007316&0.009327&0.022512&0.017317\\
\hline
128&1.21302&0.026459&0.037354&0.109024&0.124361\\
\hline
256&3.73153&0.11852&0.159628&0.483537&0.8035\\
\hline
512&14.1032&0.416178&0.601545&1.92014&10.9225\\
\hline
1024&57.606&1.87628&2.56352&7.83214&61.5952\\
\hline
2048&319.075&6.5266&9.04282&31.7321&781.52\\
\hline
\end{tabular}
\vspace{2mm}
\caption{CPU time (in seconds) for different algorithms and mesh
  sizes.}
\label{cputime}
\end{table}

\subsection{Chaotic attractor}\label{Chaotic attractor}
In this subsection, we apply our solver to a non-trivial $3$D
example. Consider the Rossler oscillator with a small random perturbations
\begin{equation}\label{Rossler_eq}
\left\{
\begin{array}{l}
dx=(-y-z)\,dt + \varepsilon\,dW_t^x\\
dy=(x+ay)\,dt + \varepsilon\,dW_t^y\\
dz=\big(b + z(x-c)\big)\,dt + \varepsilon\,dW_t^z
\end{array}
\right.,
\end{equation}
where $a=0.2$, $b = 0.2$, $c = 5.7$ $\varepsilon = 0.1$, and $W_t^x$, $W_t^y$ and $W_t^z$
are independent Wiener processes.  This system is a representative
example of chaotic ODE systems appearing in many applications of physics, biology and
engineering. Figure \ref{Rossler3d} shows a trajectory in the
corresponding deterministic system and its projection onto the
$xy$-plane.  

\begin{figure}
	\center{
		\includegraphics[width=0.49\textwidth]
		{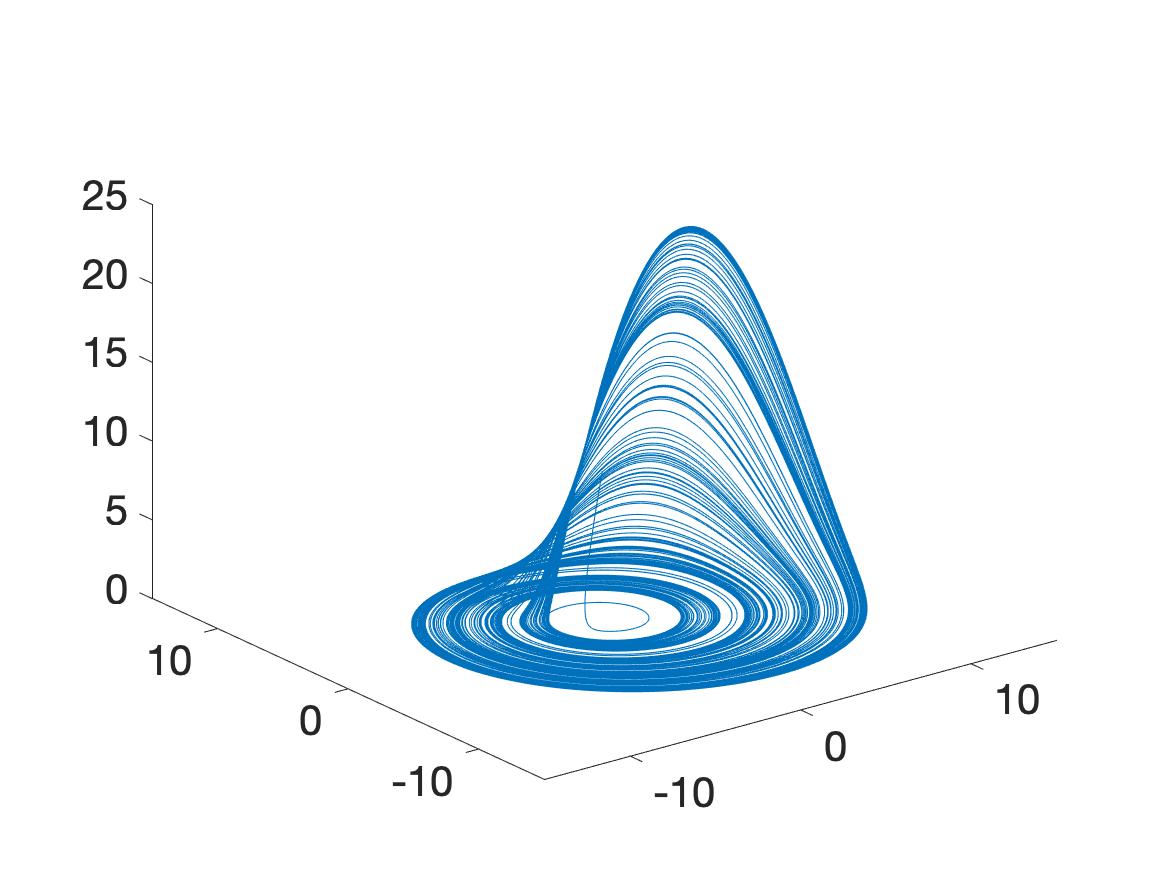}
		\includegraphics[width=0.49\textwidth]
		{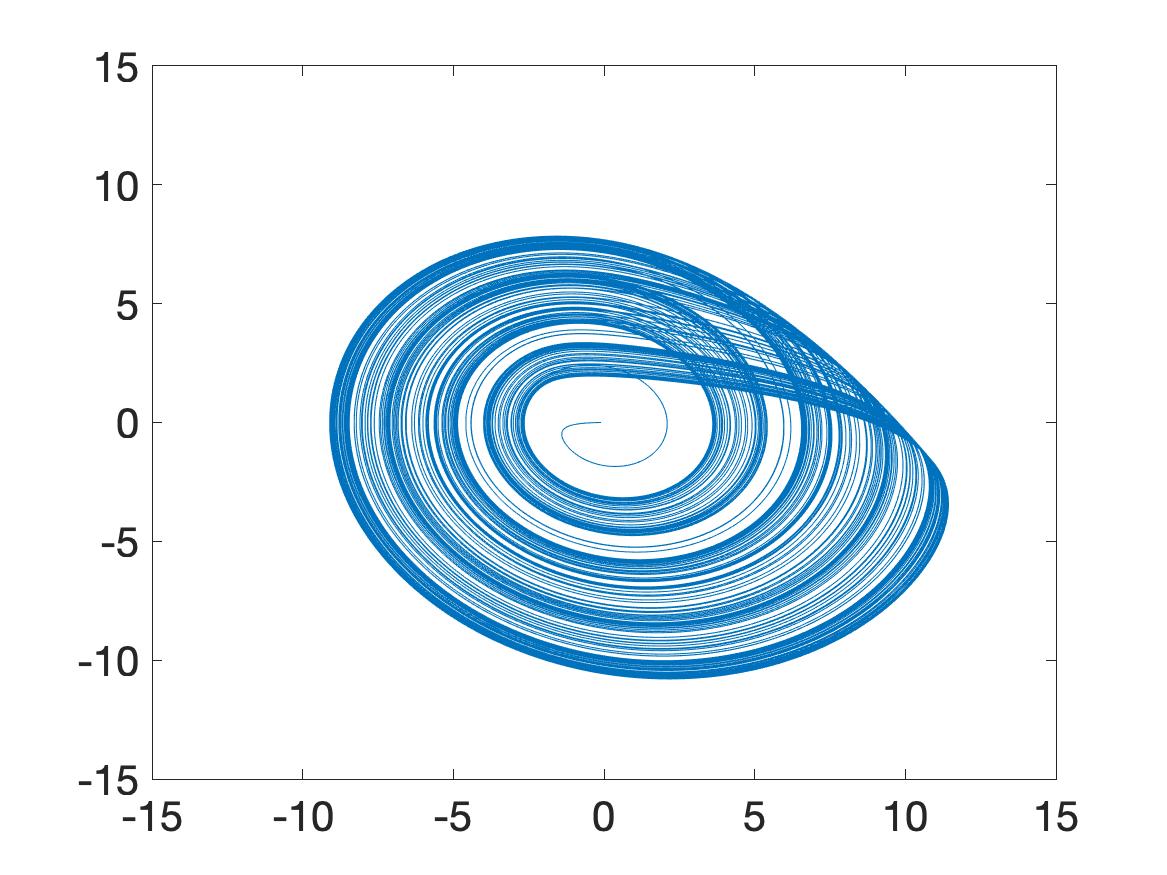}
	}
	\caption{\textbf{(Rossler)} A trajectory in the Rossler system
          \eqref{Rossler_eq} (\textbf{left}) and its projection on the
          $xy$-plane (\textbf{right}).
}\label{Rossler3d} 
\end{figure}

It is natural to imagine that the invariant density of
\eqref{Rossler_eq} has a similar shape to Figure \ref{Rossler3d}. We use the block solver together with $3$ repetitions of
the shifting blocks method on $D=[-15,15]\times[-15,15]\times[-1.5,1.5]$ with
$1024\times1024\times128$ mesh points. The grid is further divided
into $32\times 32\times 4$ blocks. The reference solution is generated
by a Monte Carlo simulation with $3.2\times10^{10}$ samples. Four
``slices'' of the solution, as seen in Figure \ref{Rossler}, are then projected to the $xy$-plane for the sake of
easier demonstration. Projection of the whole solution to the
$xy$-plane is shown in Figure \ref{Rossler_select}. In addition to the expected similar shape of the
distribution, we can see that many fine local structures of the
deterministic system are preserved by the invariant probability
measure. 

\begin{figure}
	\center{
		\includegraphics[width=\textwidth]
		{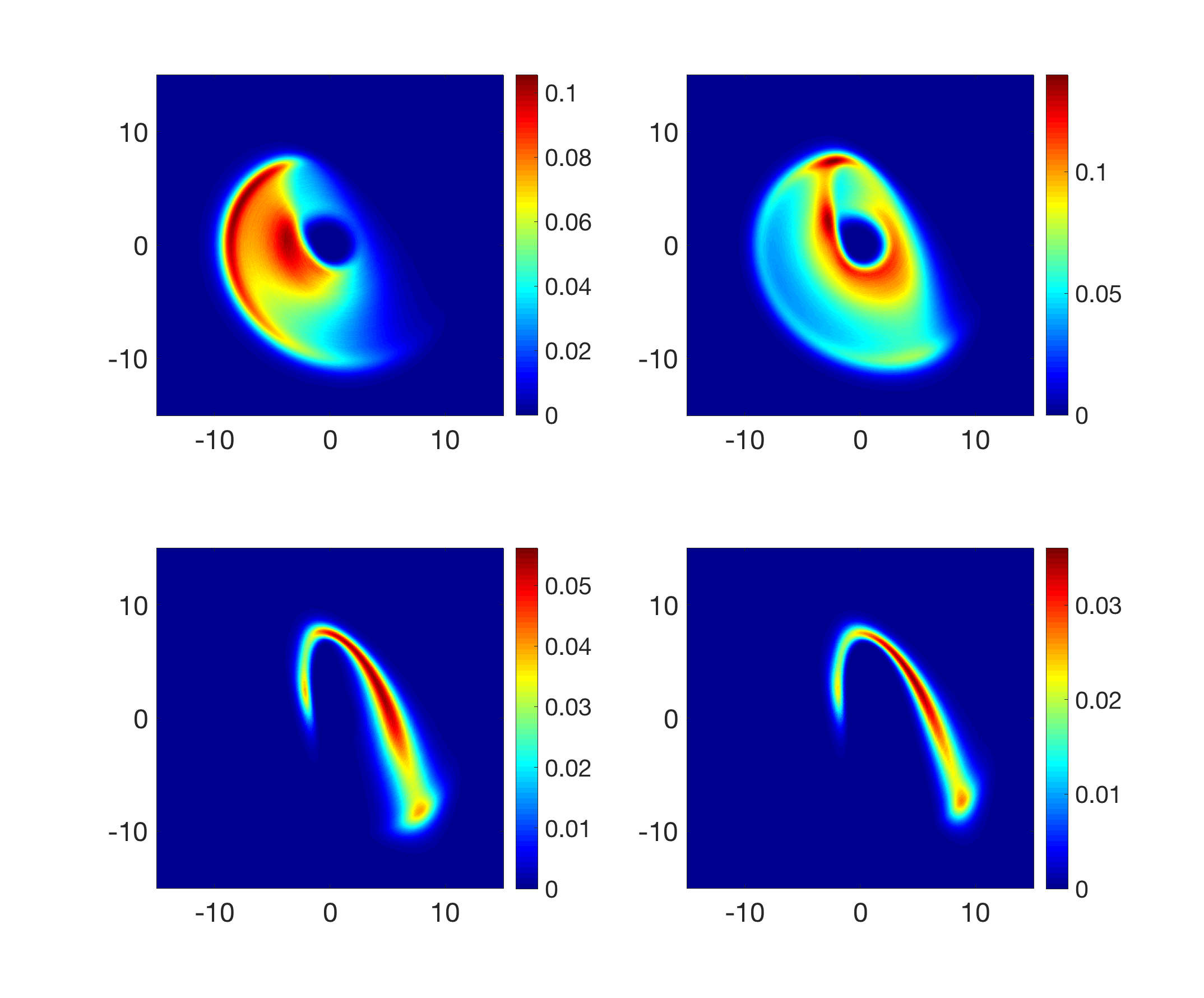}
	}
	\caption{\textbf{(Rossler)} Projections of $4$ ``slices''
          of the invariant density of the Rossler system \eqref{Rossler_eq} to the
          $xy$-plane. $z$-coordinates of $4$ slices are $[-0.09375,
          0.02344]$, $[0.02344 0.1406]$, $[0.1406, 0.2578]$, and
          $[0.2578, 0.375]$. The solution is obtained by a half-block shift
          solver on $[-15,15]\times[-15,15]\times[-1.5,2.25]$ with
          $1024\times1024\times128$ mesh points, $32\times32\times4$
          blocks, and $3.2 \times 10^{10}$ samples.}
\label{Rossler}
\end{figure}

\begin{figure}
	\center{
		\includegraphics[width=0.8\textwidth]
		{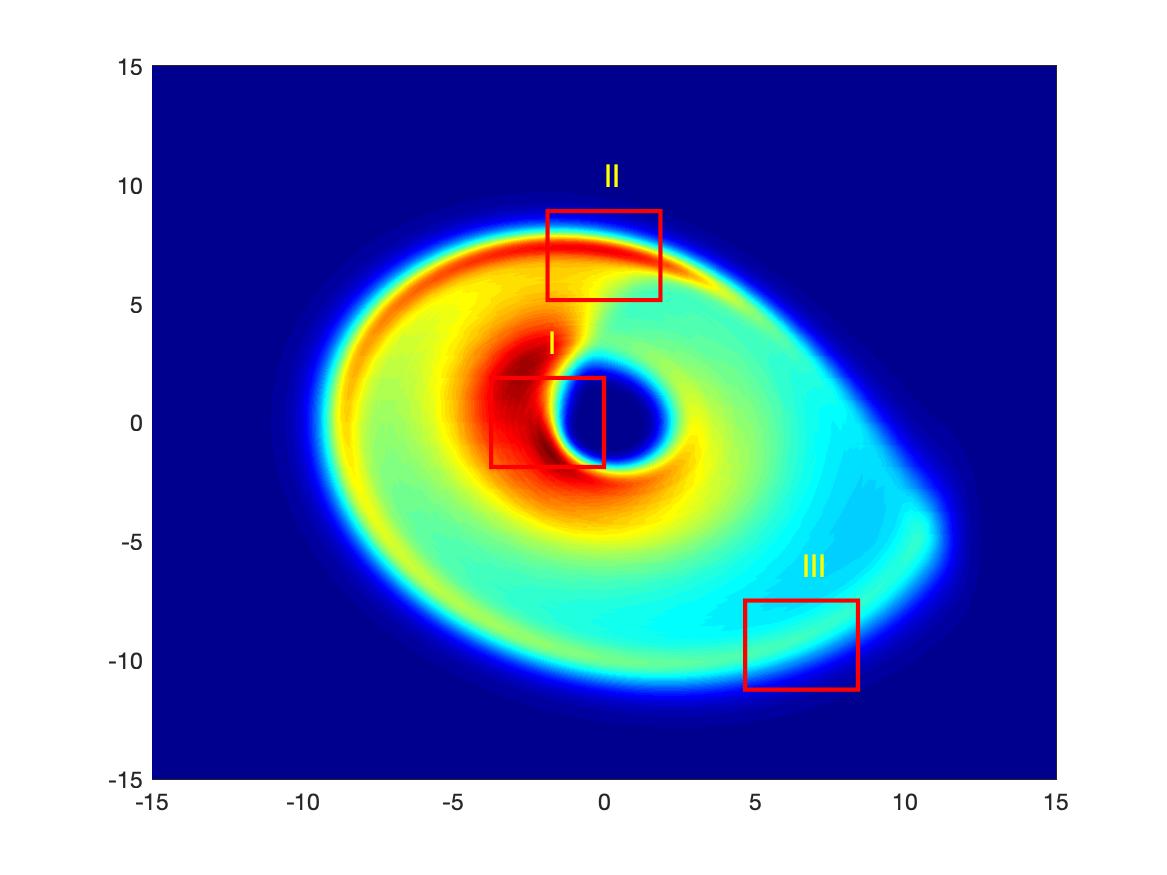}
	}
	\caption{\textbf{(Rossler)} The projection of the whole
          solution of the Rossler system \eqref{Rossler_eq} to the
          $xy$-plane. Three different local regions I-III (the red boxes) are used for
          comparing block solvers with different block sizes. See
          Figure \ref{Rossler_local} also.}
\label{Rossler_select}
\end{figure}

To demonstrate the performance of our algorithm, we apply the
data-driven solver without blocks to three local regions with
different characteristics (see 
Figure \ref{Rossler_select}). In each region, we use the data-driven
solver on $128\times128\times128$ mesh points without dividing the domain into blocks. 

In Region I, $[-15/4,0]\times[-15/8,15/8]$, the projection of the
solution has both dense and sparse parts that are clearly divided. In 
the first figure of \ref{Rossler_local}, we can see that a similar
resolution is preserved when using much smaller block sizes. Both
solutions provide high resolution to demonstrate the influence of strong chaos on the
invariant distribution. The only difference is the local solver with
smaller blocks has higher error on the left and bottom boundary,
because the half-shift method does not touch this part. The discrete
$L^2$ norm of the difference between the restriction of global
solution on Region I and the local solution is $\varepsilon_{\text{I}}
\approx 0.0032$. In Region II, $[-15/8,15/8]\times[165/32,285/32]$, the solution
includes an outer ``ring'' with high density. Outside this ``ring'', the density function decays quickly. We can see that both
solutions show the decay of the density around this ring. The discrete
$L^2$ norm of this difference between the two solutions in Region II
is $\varepsilon_{\text{II}} \approx 0.0028$. In Region III, $[75/16,135/16]\times[-45/4,15/2]$, the local solution
has much lower density. The local solver is still accurate when the entries of
$\textbf{v}$ are much smaller. The discrete
$L^2$ norm of the difference between the global solution and the local
solution in this region is $\varepsilon_{\text{III}} \approx
0.0018$. 

Overall, the solution from the block solver has little difference from
the one obtained over a large mesh. And the
solver can provide desired resolution in both settings. Empirically,
we find that a block size of $30-35$ is a good balance of
performance and accuracy for most 2D and 3D problems.

\begin{figure}
	\center{
		\includegraphics[width=0.32\textwidth]
		{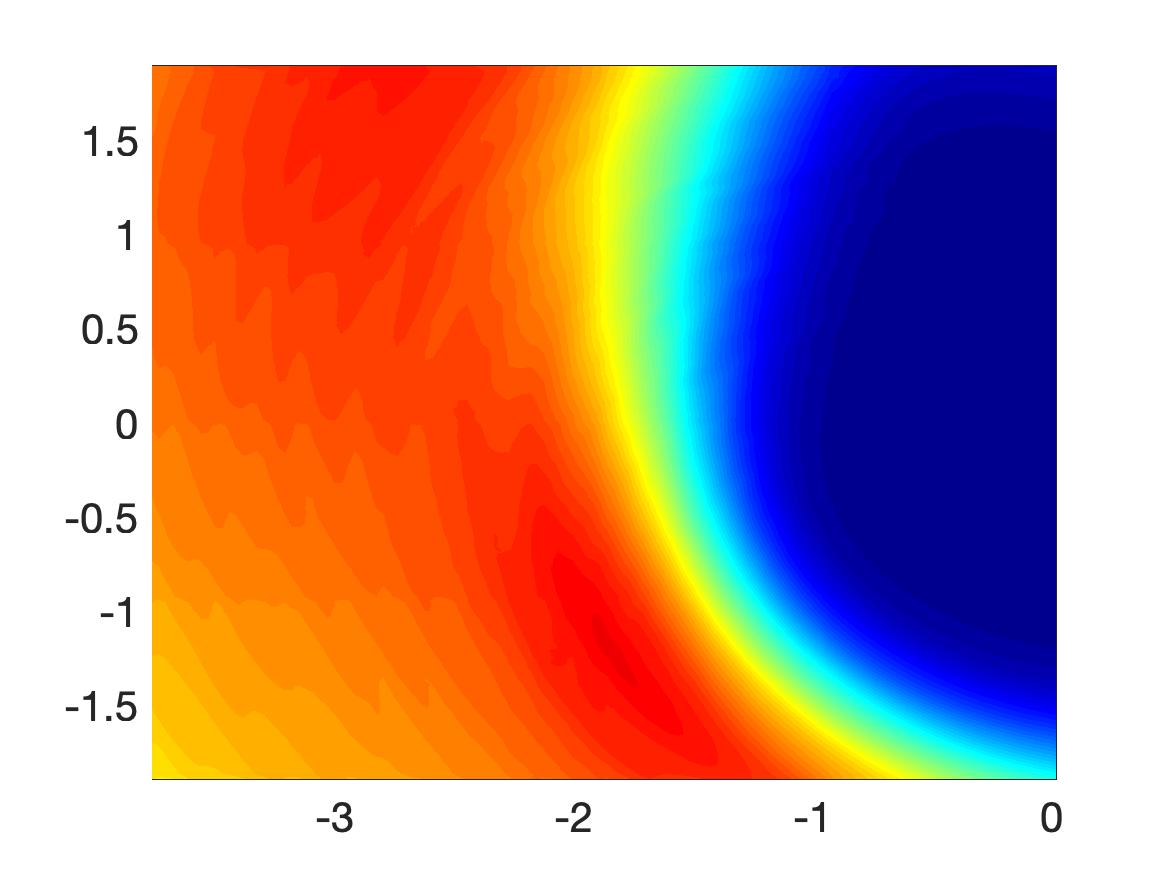}
		\includegraphics[width=0.32\textwidth]
		{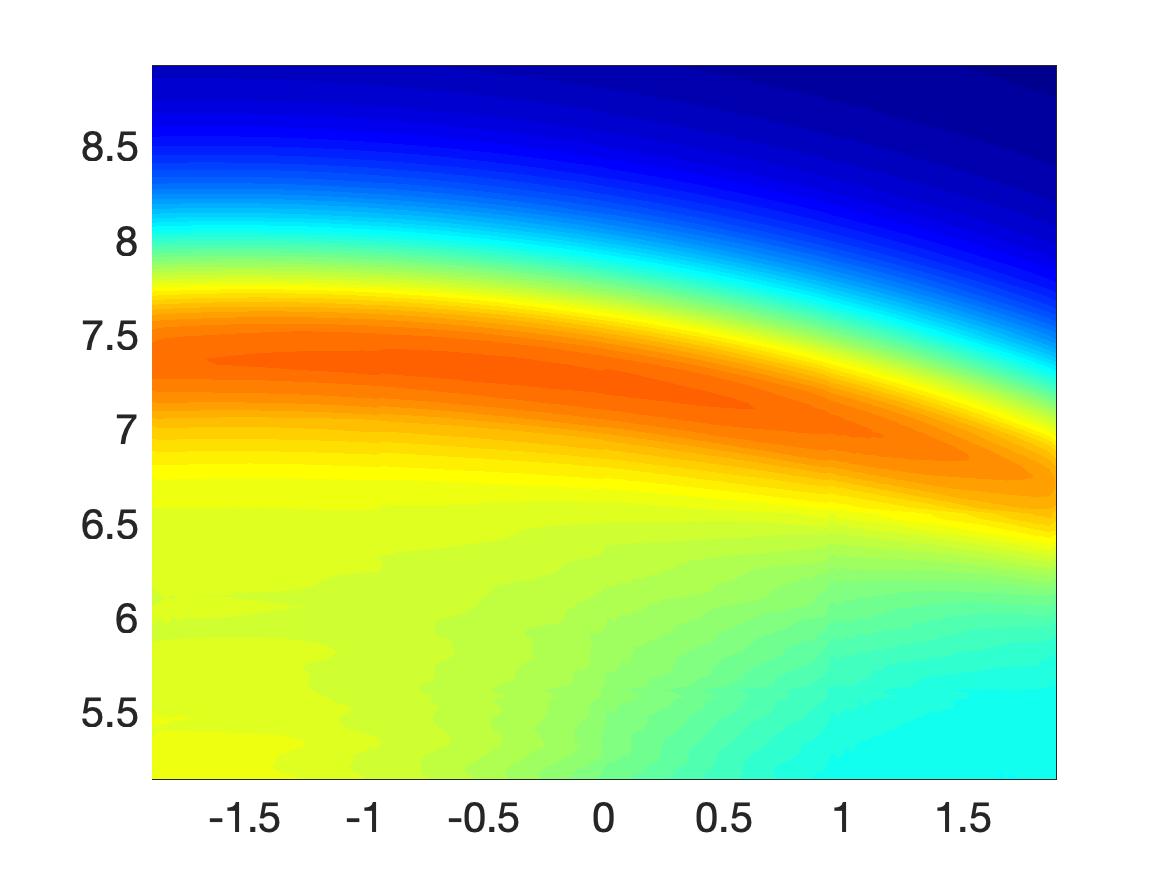}
		\includegraphics[width=0.32\textwidth]
		{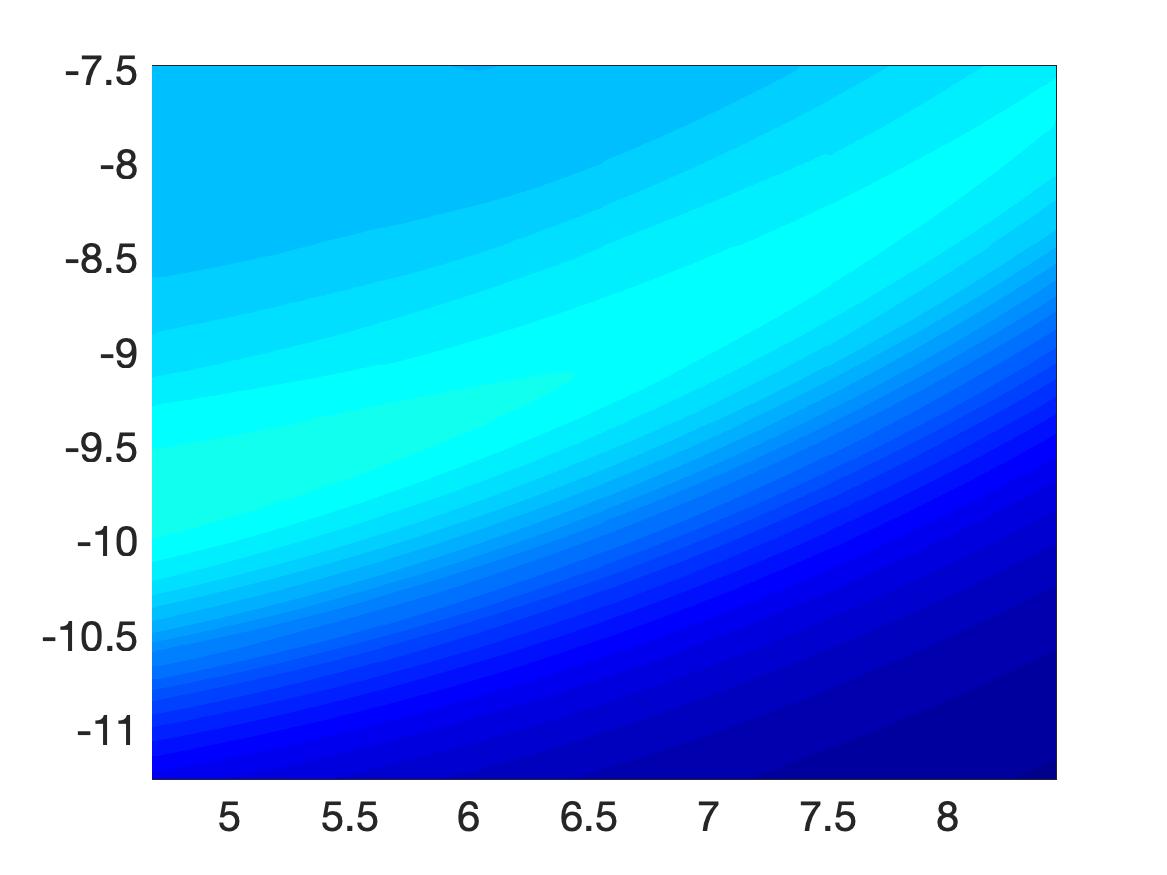}
		\includegraphics[width=0.32\textwidth]
		{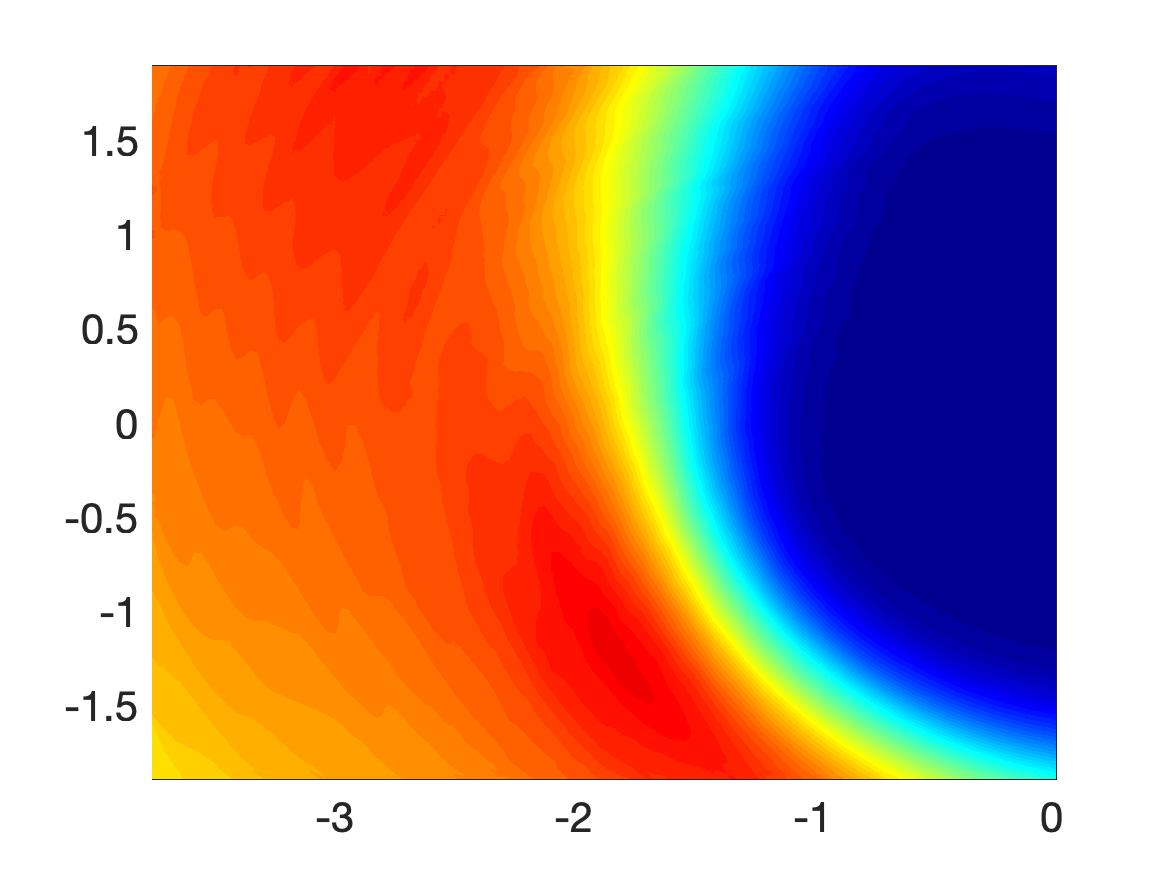}
		\includegraphics[width=0.32\textwidth]
		{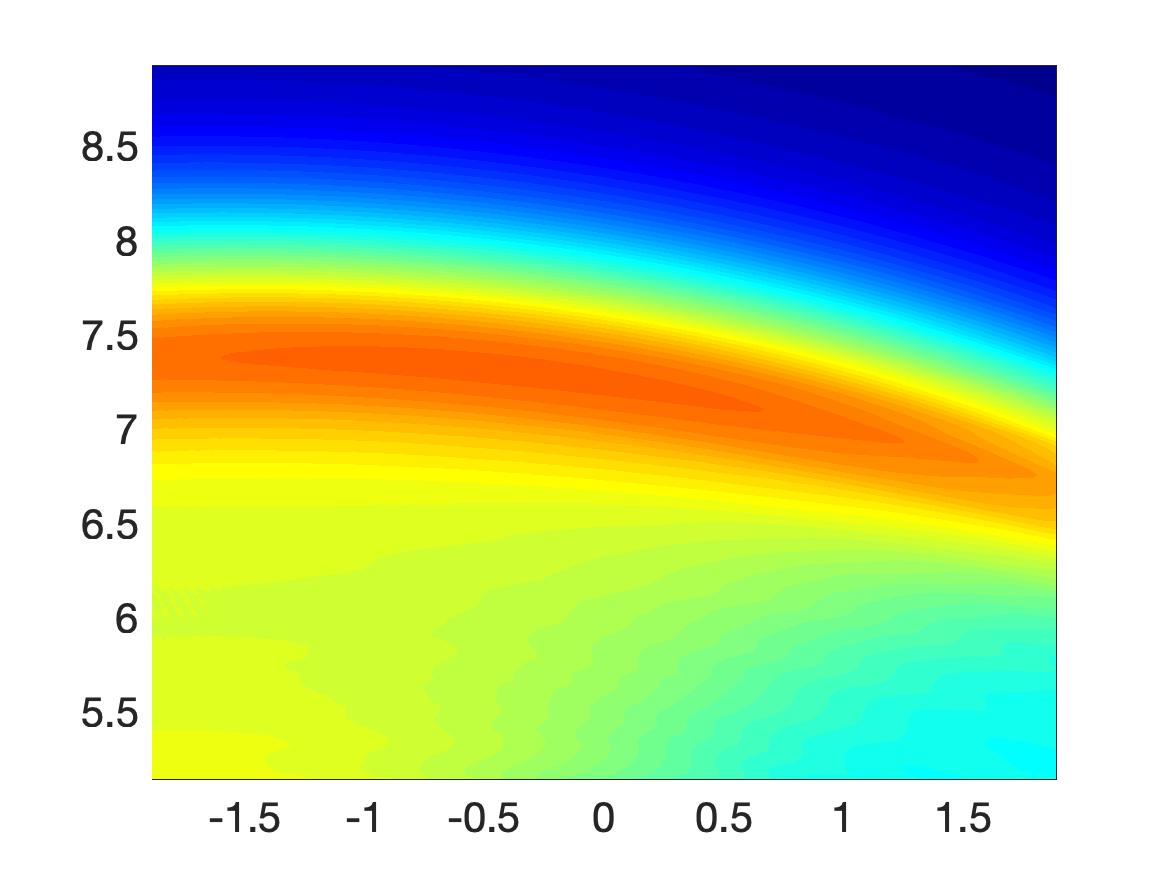}
		\includegraphics[width=0.32\textwidth]
		{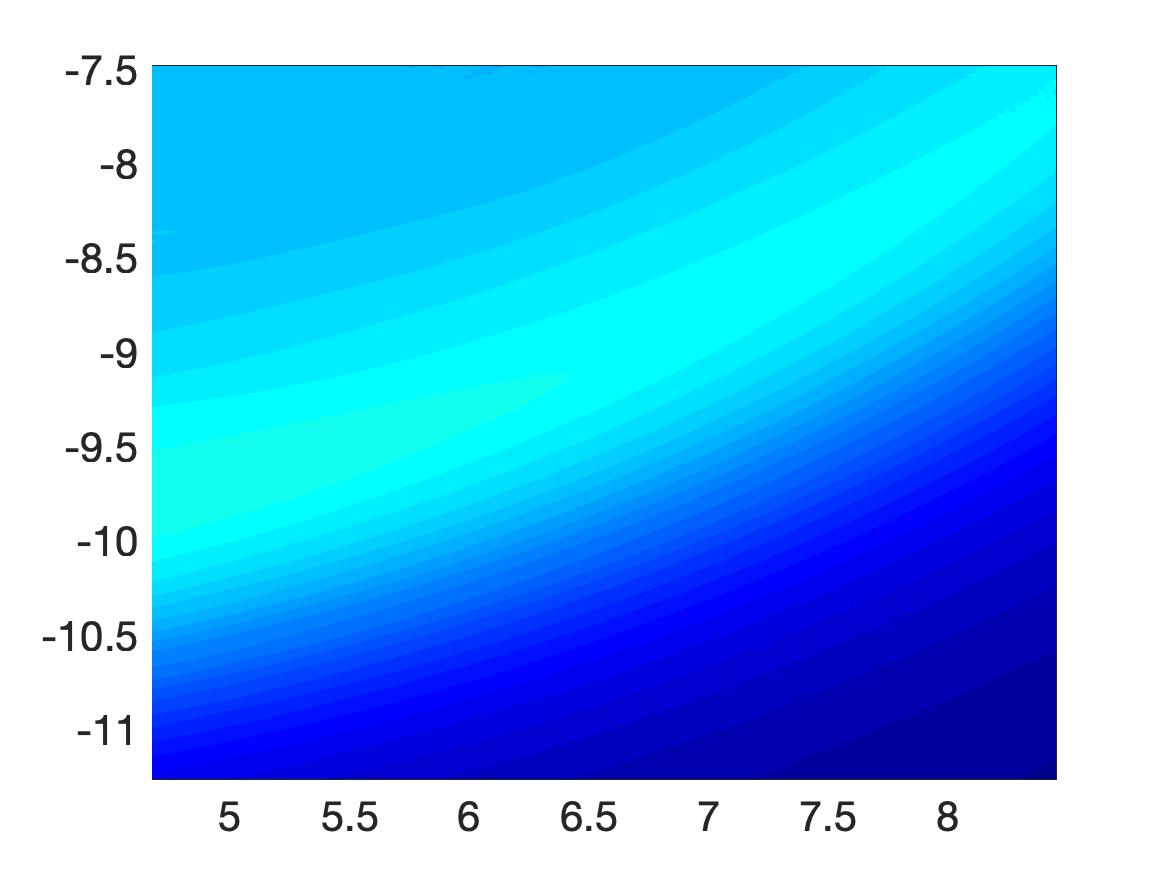}		
	}
	\caption{\textbf{(Rossler)} The local restrictions of the
          global solution in Region I--III (the first row), and the
          projections onto the $xy$-plane of solutions by the local
          solver in the three local regions (the second
          row). }\label{Rossler_local}
\end{figure}

\subsection{Mixed mode oscillation}\label{Mixed mode oscillation}

In this example, we consider another non-trivial $3$D system of mixed mode oscillation (MMO) with small random perturbations
\begin{equation}\label{MMO_eq}
\left\{
\begin{array}{l}
dx=\frac{1}{\eta}(y - x^2 - x^3)\,dt + \varepsilon\,dW_t^x\\
dy=(z-x)\,dt + \varepsilon\,dW_t^y\\
dz=(- \nu - ax - by - cz)\,dt + \varepsilon\,dW_t^z
\end{array}
\right.,
\end{equation}
where $\eta=0.01$, $\nu=0.0072168$, $a=-0.3872$, $b = -0.3251$, $c =
1.17$, and $W_t^x$, $W_t^y$ and $W_t^z$ are independent Wiener
processes. The strength of noise is chosen to be $\varepsilon=0.1$. Figure
\ref{MMO3d} provides one trajectory of the corresponding deterministic
system and its projection on the $xy$-plane. The deterministic part of
equation \eqref{MMO_eq} has a critical manifold $y = x^{2} + x^{3}$,
at which the derivative of the fast variable vanishes. We can see that
oscillations with different amplitudes occur near the fold of the
critical manifold, where the attracting and repelling sheet of the
critical manifold meet. This is called the mixed mode oscillation
(MMO) \cite{desroches2012mixed}. The mechanism of mixed mode oscillations is similar as
that of the canard explosion, which means the trajectory can follow
the unstable sheet of the critical manifold for some time \cite{guckenheimer2005canards}. It was
observed in \cite{li2018data} that the canard explosion can be destroyed by
a small random perturbation. This motivates us to explore the
characteristics of MMO under random perturbations. 

We again use the half-block shift solver with $2048\times512\times256$
mesh points, $64\times16\times8$ blocks and $10^9$ samples to get the
invariant measure on
$D=[-1.5,0.5]\times[-0.15,0.35]\times[-0.1,0.15]$. The numerical
result is still projected to the $xy$-plane (see Figure \ref{MMO}). We can see that
the invariant measure is mainly supported by the neighborhood of the
stable sheets of the critical manifold. Deterministic oscillations with small
amplitude are eliminated by the random perturbation. In other words,
similar to the canard explosion, MMO can not survive a small
random perturbation. The mechanism of this phenomenon is still
not clear. It is also not known how small the noise
should be in order to see MMO in equation \eqref{MMO_eq}.

To corroborate the performance of the solver on local regions, in this
example, we apply it to four `$z$-layers', that is, the region in the
phase space of the form $[-1.5,0.5]\times[-0.15,0.35]\times I$, where
$I=[-0.1,-0.1+d], [-0.05,-0.05+d], [0.05,0.05+d]$, and $[0.1,0.1+d]$
respectively with $d=1/32$. In each layer, we apply an iterated
shifting blocks solver with $2048\times512\times32$ mesh points,
$64\times16\times1$ blocks, and $10^9$ samples.  

Figure \ref{MMO_local} shows the invariant distribution in these four local
layers when projected to the $xy$-plane. We can see the invariant density
function on each $z$-layer. Similar as in Figure \ref{MMO}, most
invariant density concentrates at two stable sheets of the invariant
manifold, and no small amplitude oscillations can be seen from the
invariant probability density function.

\begin{figure}
	\center{
		\includegraphics[width=0.49\textwidth]
		{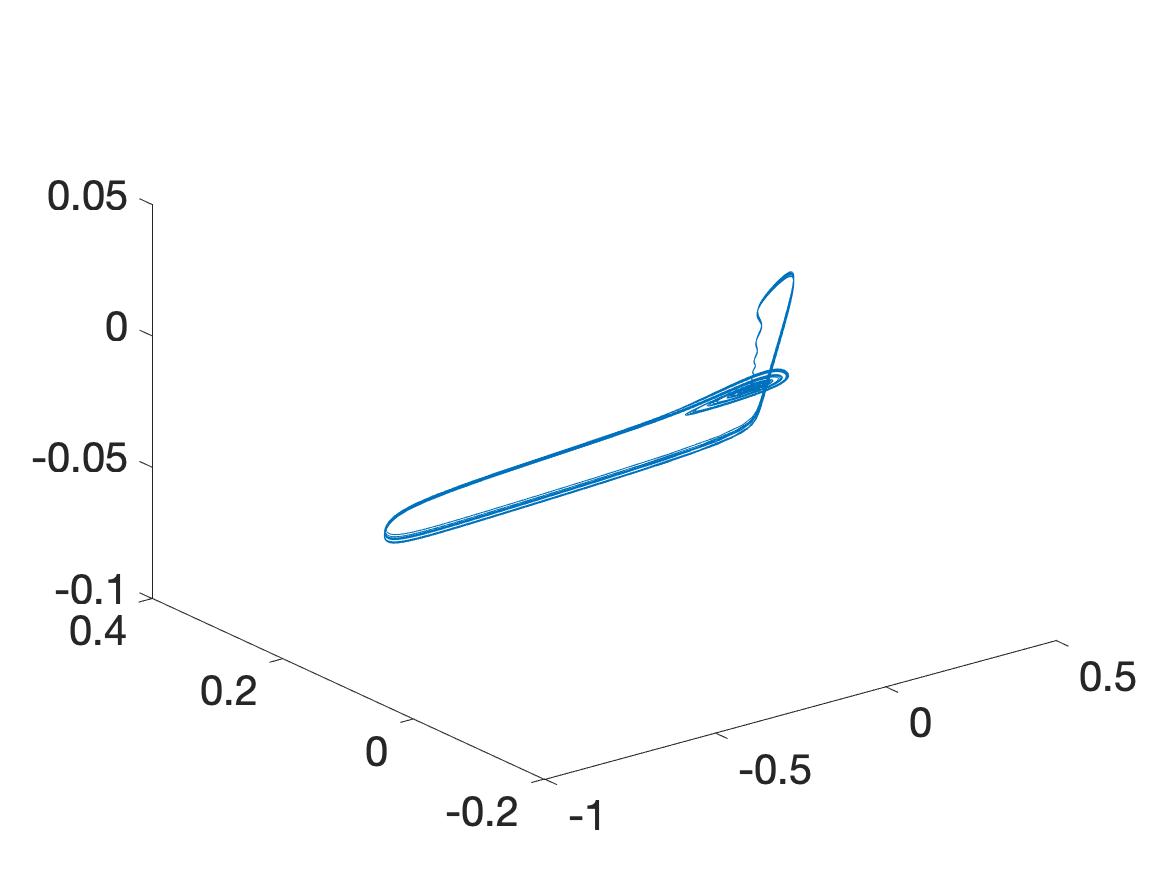}
		\includegraphics[width=0.49\textwidth]
		{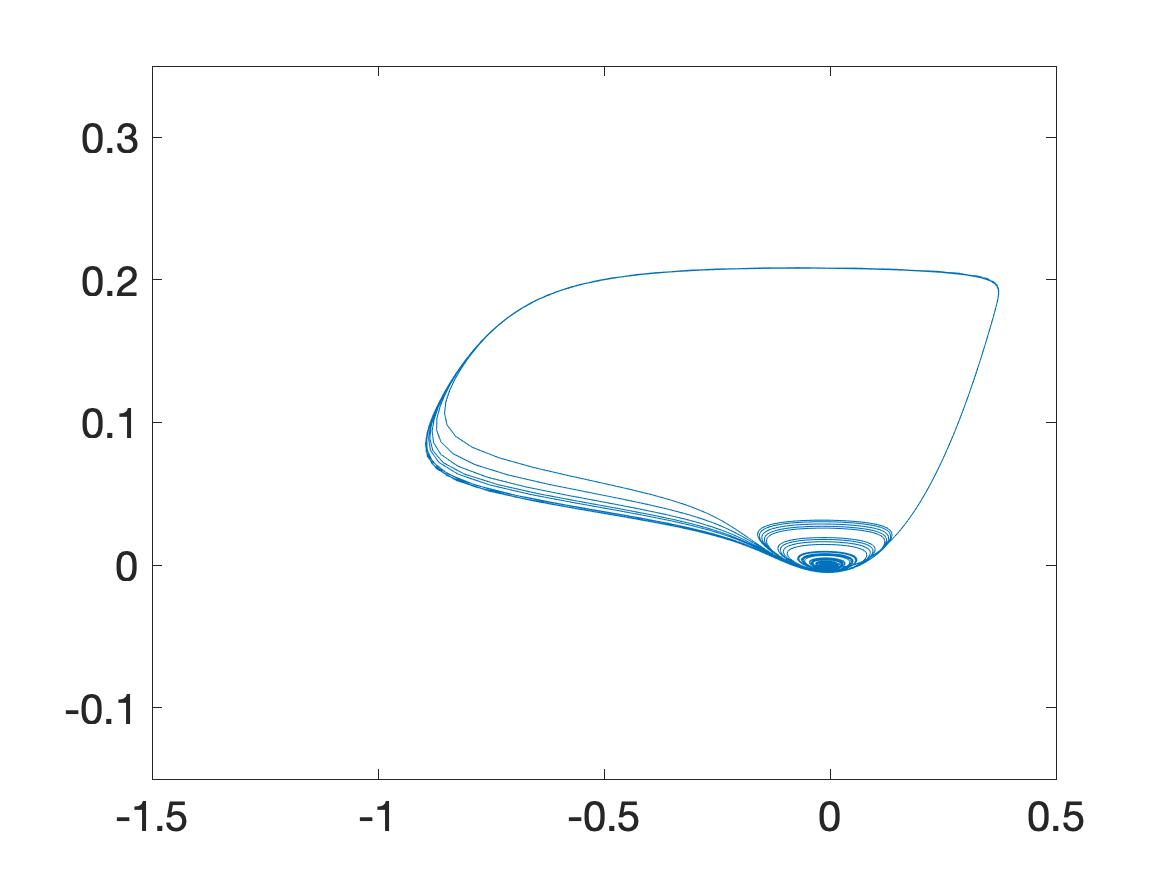}
	}
	\caption{\textbf{(MMO)} A trajectory in the system of mixed mode oscillation \eqref{MMO_eq} (\textbf{left}) and its projection on the $xy$-plane (\textbf{right}).}\label{MMO3d}
\end{figure}

\begin{figure}
	\center{
		\includegraphics[width=0.6\textwidth]
		{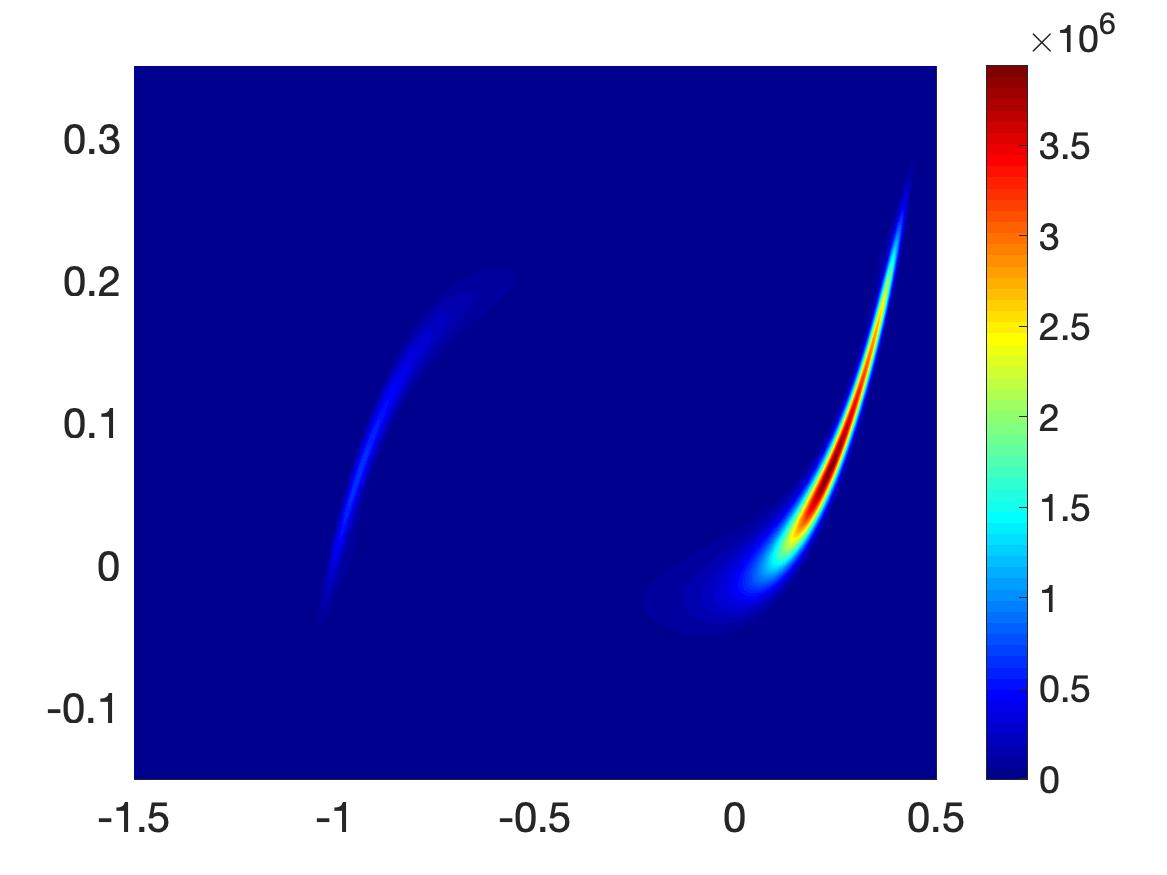}
	}
	\caption{\textbf{(MMO)} The projection of the invariant density of the system \eqref{MMO_eq} of mixed mode oscillation onto the $xy$-plane by a half-block shift solver on $[-1.5,0.5]\times[-0.15,0.35]\times[-0.1,0.15]$ with $2048\times512\times256$ mesh points, $32\times32\times32$ blocks and $10^9$ samples.}\label{MMO}
\end{figure}

\begin{figure}
	\center{
		\includegraphics[width=0.49\textwidth]
		{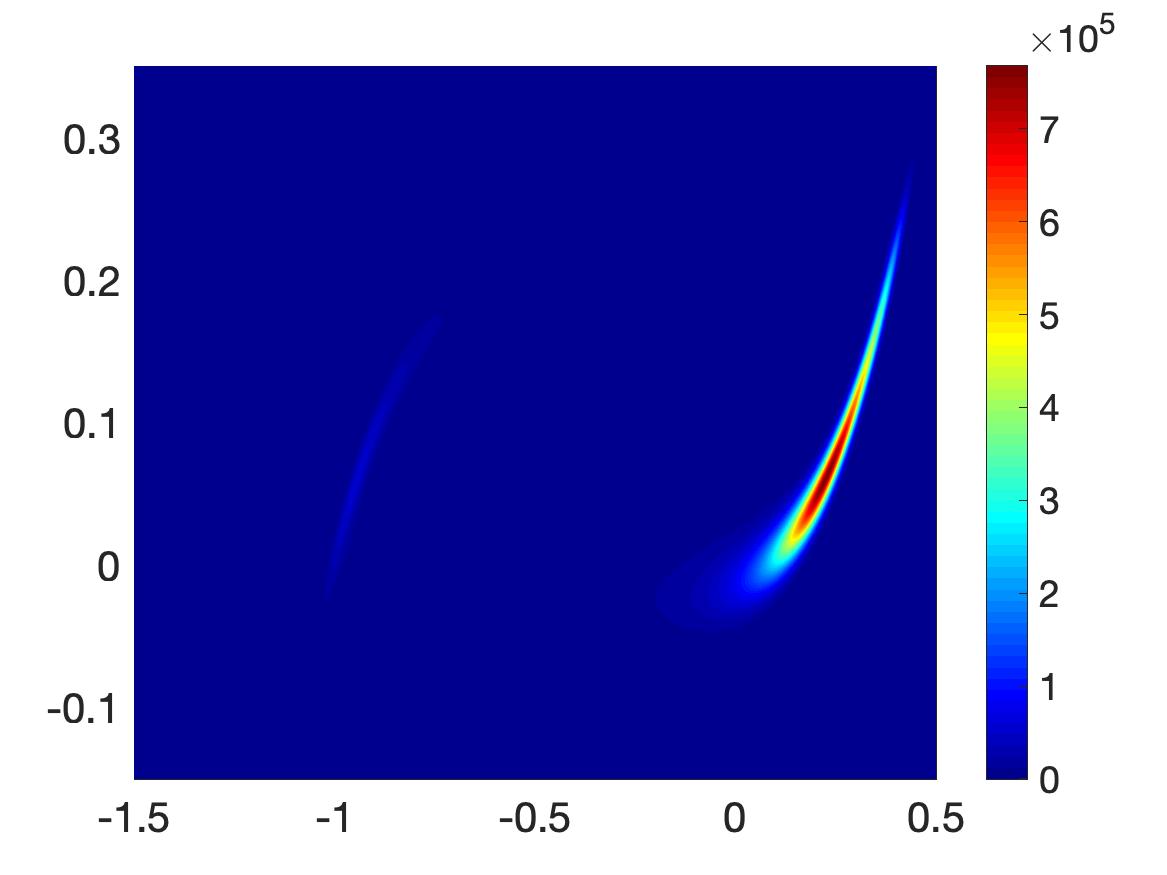}
		\includegraphics[width=0.49\textwidth]
		{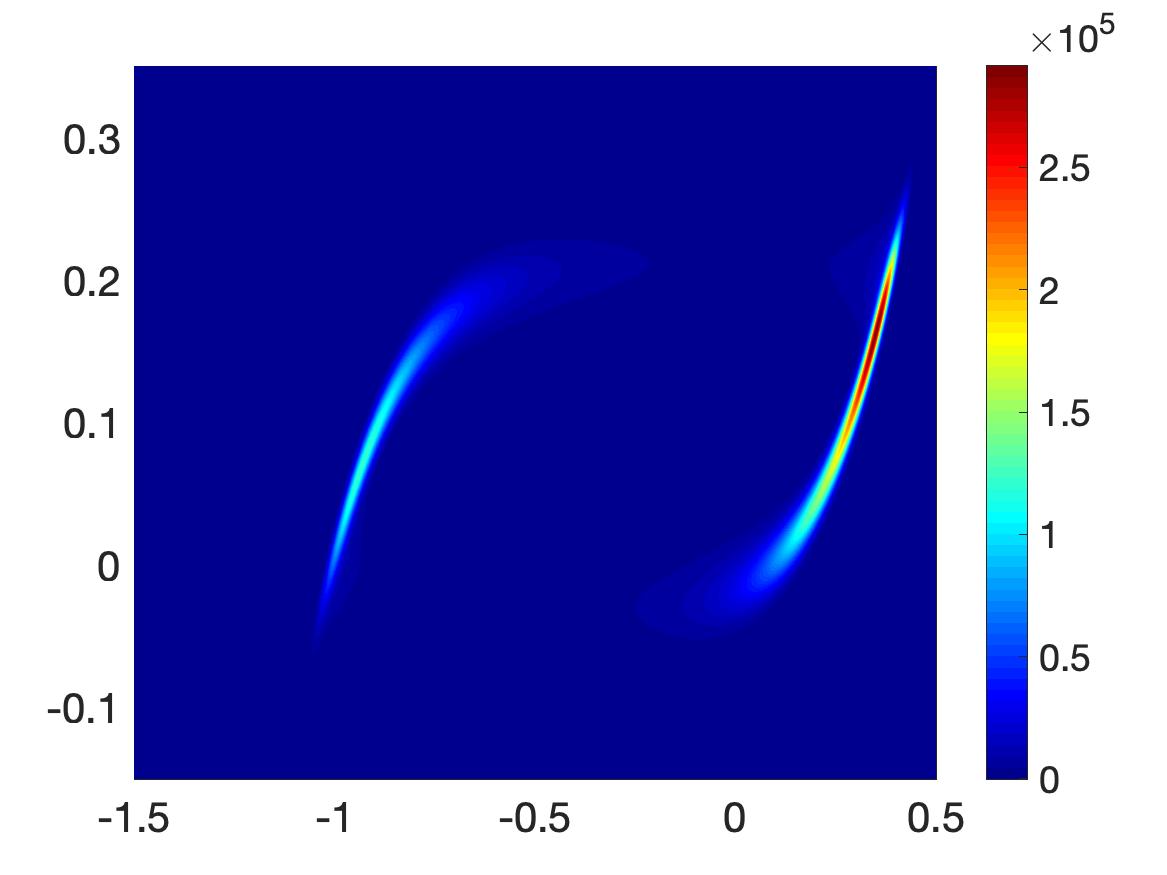}
		\includegraphics[width=0.49\textwidth]
		{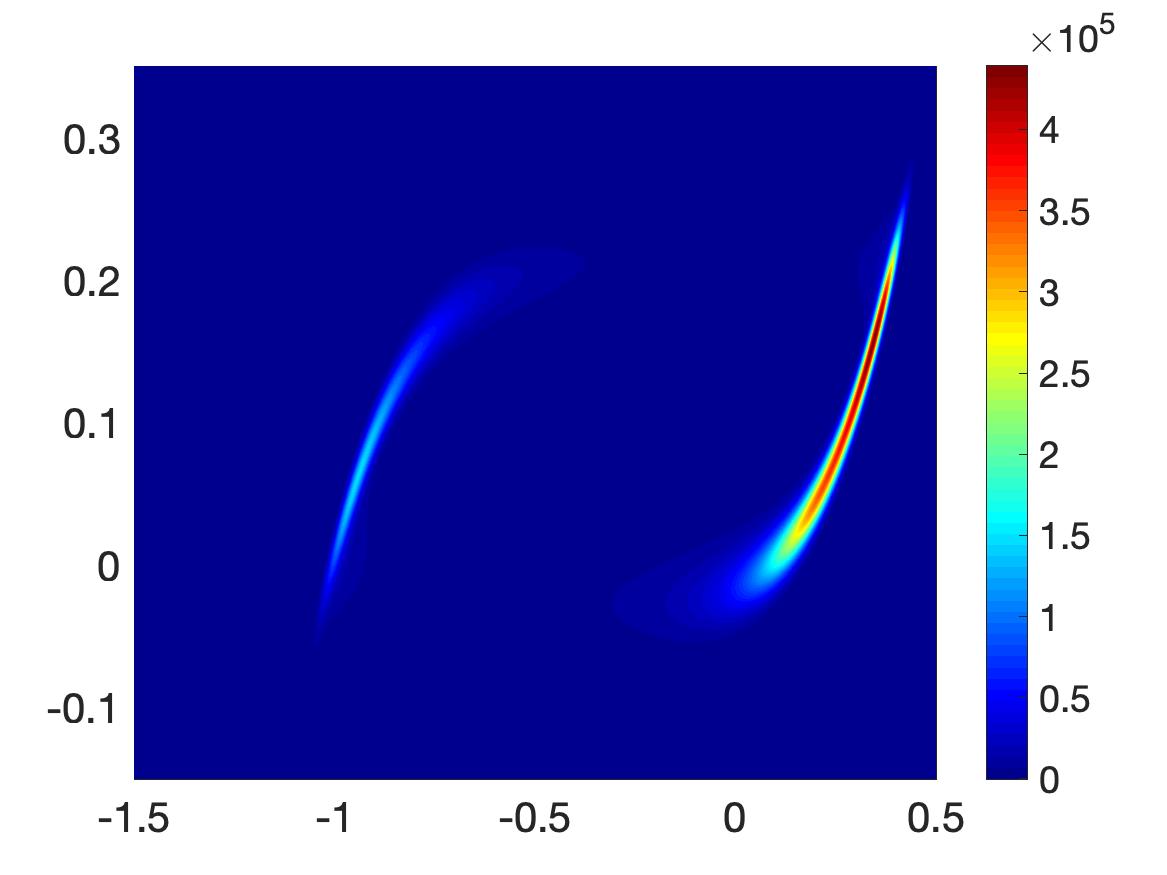}
		\includegraphics[width=0.49\textwidth]
		{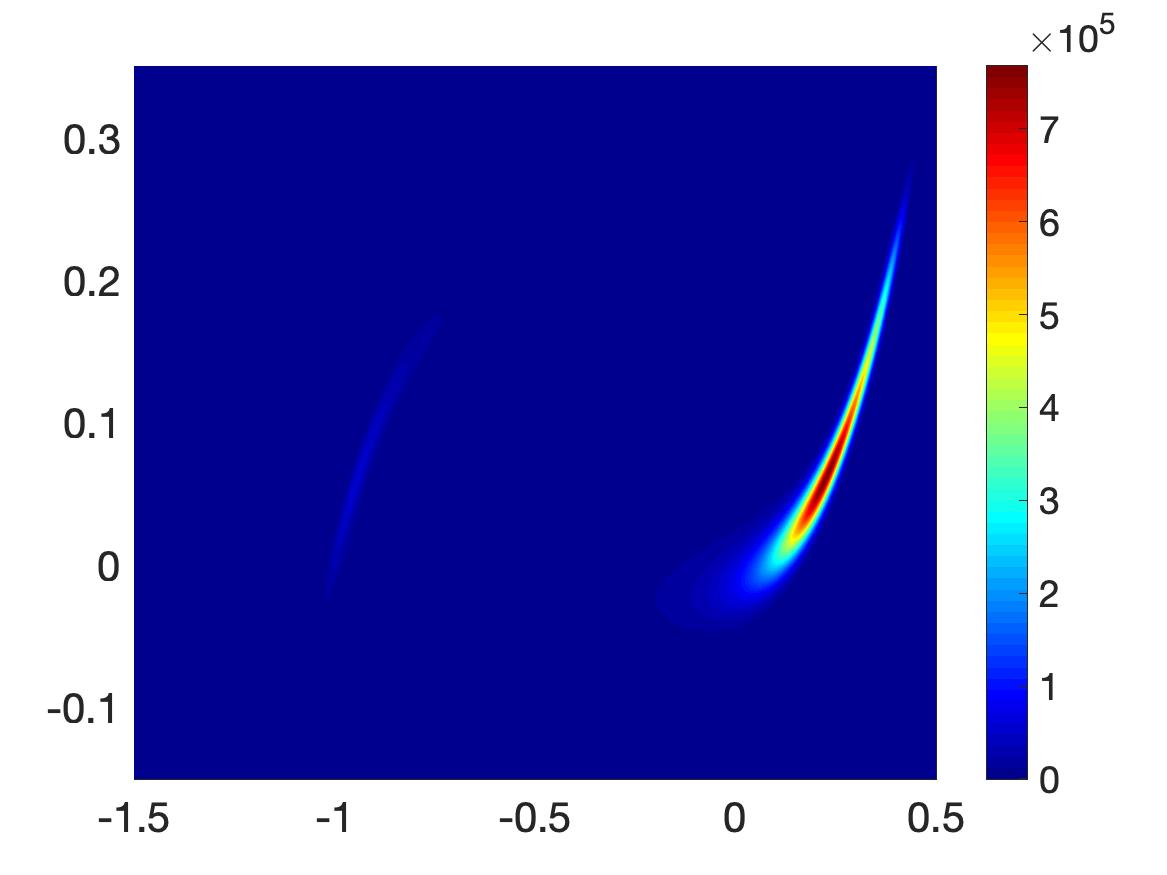}
	}
	\caption{\textbf{(MMO)} The projections onto the $xy$-plane of solutions by the local solver in the four $z$-layers).}\label{MMO_local}
\end{figure}

\section{Conclusion}
\label{conclusion}

A hybrid method for computing the invariant probability measure of the
Fokker-Planck equation was proposed in \cite{li2018data}. The key idea is
to generate a reference solution from Monte Carlo simulation to partially
replace the role of boundary conditions. In this paper, we rigorously
proved the convergence of this hybrid method. The concentration of
error is also investigated analytically and numerically. We found that the error
tends to concentrate on the boundary of the numerical domain,
which makes the empirical performance much better than our theoretical
result. Motivated by these results and the divide-and-conquer
strategy, we proposed a block version of this hybrid method. It
dramatically reduces the computational cost for problems up to
dimension 4. This method makes the computation of invariant
probability measures possible for many stochastic differential
equations arising in different fields, especially for researchers
with limited computing resources. Finally, to repair the interface
error appearing at the interface between adjacent
blocks, two different methods are proposed and tested with several
numerical examples.  

The block solver studied in this paper can be extended into several
directions. A natural extension is the time-dependent Fokker-Planck
equations. As discussed in \cite{li2018data}, one only needs to slightly
modify the optimization problem \eqref{opt} to solve a
time-dependent Fokker-Planck equation. This data-driven framework also
works for other PDEs with available data from stochastic simulations, such as
reaction-diffusion equations. It is well known that a chemical
reaction system with diffusions can be computed by both the stochastic simulation algorithm (SSA)
and the reaction-diffusion equation. This is similar to the case of
the Fokker-Planck equation. In addition, mesh-free version of this
block solver can be developed to solve higher dimensional
problems. Some high-dimensional sampling methods \cite{chen2017beating,
  chen2018efficient} can be adopted to improve the quality of
sampling.


\bibliography{myref}{}

\providecommand{\bysame}{\leavevmode\hbox to3em{\hrulefill}\thinspace}
\providecommand{\MR}{\relax\ifhmode\unskip\space\fi MR }
\providecommand{\MRhref}[2]{%
  \href{http://www.ams.org/mathscinet-getitem?mr=#1}{#2}
}
\providecommand{\href}[2]{#2}
\begin{thebibliography}{10}

\bibitem{bogachev2001generalization}
V.~Bogachev and M.~R{\"o}ckner, \emph{A generalization of {K}hasminskii's
  theorem on the existence of invariant measures for locally integrable
  drifts}, Theory of Probability and its Applications \textbf{45} (2001), 363.

\bibitem{chen2017beating}
Nan Chen and Andrew~J Majda, \emph{Beating the curse of dimension with accurate
  statistics for the {F}okker-{P}lanck equation in complex turbulent systems},
  Proceedings of the National Academy of Sciences \textbf{114} (2017), no.~49,
  12864--12869.

\bibitem{chen2018efficient}
\bysame, \emph{Efficient statistically accurate algorithms for the
  {F}okker-{P}lanck equation in large dimensions}, Journal of Computational
  Physics \textbf{354} (2018), 242--268.

\bibitem{day1985some}
Martin~V Day and Thomas~A Darden, \emph{Some regularity results on the
  {V}entcel-{F}reidlin quasi-potential function}, Applied Mathematics and
  Optimization \textbf{13} (1985), no.~1, 259--282.

\bibitem{desroches2012mixed}
Mathieu Desroches, John Guckenheimer, Bernd Krauskopf, Christian Kuehn, Hinke~M
  Osinga, and Martin Wechselberger, \emph{Mixed-mode oscillations with multiple
  time scales}, Siam Review \textbf{54} (2012), no.~2, 211--288.

\bibitem{er2011methodology}
Guo-Kang Er, \emph{Methodology for the solutions of some reduced
  {F}okker-{P}lanck equations in high dimensions}, Annalen der Physik
  \textbf{523} (2011), no.~3, 247--258.

\bibitem{er2012state}
Guo-Kang Er and Vai~Pan Iu, \emph{State-space-split method for some generalized
  {F}okker-{P}lanck-{K}olmogorov equations in high dimensions}, Physical Review
  E \textbf{85} (2012), no.~6, 067701.

\bibitem{freidlin1998random}
Mark~Iosifovich Freidlin and Alexander~D Wentzell, \emph{Random perturbations},
  Random Perturbations of Dynamical Systems, Springer, 1998, pp.~15--43.

\bibitem{guckenheimer2005canards}
John Guckenheimer and Radu Haiduc, \emph{Canards at folded nodes}, Moscow
  Mathematical Journal \textbf{5} (2005), no.~1, 91--103.

\bibitem{hairer2006ergodicity}
Martin Hairer and Jonathan~C Mattingly, \emph{Ergodicity of the 2d
  {N}avier-{S}tokes equations with degenerate stochastic forcing}, Annals of
  Mathematics (2006), 993--1032.

\bibitem{huang2015steady}
Wen Huang, Min Ji, Zhenxin Liu, and Yingfei Yi, \emph{Steady states of
  {F}okker-{P}lanck equations: I. existence}, Journal of Dynamics and
  Differential Equations \textbf{27} (2015), no.~3-4, 721--742.

\bibitem{karatzas2012brownian}
Ioannis Karatzas and Steven Shreve, \emph{Brownian motion and stochastic
  calculus}, vol. 113, Springer Science \& Business Media, 2012.

\bibitem{khasminskii2011stochastic}
Rafail Khasminskii, \emph{Stochastic stability of differential equations},
  vol.~66, Springer Science \& Business Media, 2011.

\bibitem{lelievre2016partial}
Tony Lelievre and Gabriel Stoltz, \emph{Partial differential equations and
  stochastic methods in molecular dynamics}, Acta Numerica \textbf{25} (2016),
  681--880.

\bibitem{li2018data}
Yao Li, \emph{A data-driven method for the steady state of randomly perturbed
  dynamics}, Communications in Mathematical Sciences, accepted (2019).

\bibitem{li2016systematic}
Yao Li and Yingfei Yi, \emph{Systematic measures of biological networks {I}:
  Invariant measures and entropy}, Communications on Pure and Applied
  Mathematics \textbf{69} (2016), no.~9, 1777--1811.

\bibitem{majda2001mathematical}
Andrew~J Majda, Ilya Timofeyev, and Eric Vanden~Eijnden, \emph{A mathematical
  framework for stochastic climate models}, Communications on Pure and Applied
  Mathematics: A Journal Issued by the Courant Institute of Mathematical
  Sciences \textbf{54} (2001), no.~8, 891--974.

\bibitem{mattingly2010convergence}
Jonathan~C Mattingly, Andrew~M Stuart, and Michael~V Tretyakov,
  \emph{Convergence of numerical time-averaging and stationary measures via
  {P}oisson equations}, SIAM Journal on Numerical Analysis \textbf{48} (2010),
  no.~2, 552--577.

\bibitem{meyn2012markov}
Sean~P Meyn and Richard~L Tweedie, \emph{Markov chains and stochastic
  stability}, Springer Science \& Business Media, 2012.

\bibitem{oksendal2003stochastic}
Bernt {\O}ksendal, \emph{Stochastic differential equations}, Stochastic
  differential equations, Springer, 2003, pp.~65--84.

\bibitem{risken1996fokker}
Hannes Risken, \emph{Fokker-{P}lanck equation}, The Fokker-Planck Equation,
  Springer, 1996, pp.~63--95.

\bibitem{sun2015numerical}
Yifei Sun and Mrinal Kumar, \emph{A numerical solver for high dimensional
  transient {F}okker-{P}lanck equation in modeling polymeric fluids}, Journal
  of Computational Physics \textbf{289} (2015), 149--168.

\bibitem{von2000calculation}
Utz von Wagner and Walter~V Wedig, \emph{On the calculation of stationary
  solutions of multi-dimensional {F}okker-{P}lanck equations by orthogonal
  functions}, Nonlinear Dynamics \textbf{21} (2000), no.~3, 289--306.

\bibitem{zeeman1988stability}
EC~Zeeman, \emph{Stability of dynamical systems}, Nonlinearity \textbf{1}
  (1988), no.~1, 115.

\end{thebibliography}
\bibliographystyle{amsplain}
\end{document}